\documentclass[amsfonts]{amsart}
\usepackage{amssymb, amsfonts , tikz-cd, bm}
\usepackage[all,arc]{xy}
\usepackage{enumitem}
\usepackage{mathrsfs}
\usepackage{microtype}
\usepackage{amsaddr}
\usepackage{amsthm,etoolbox}
\usepackage{eucal}

\usepackage{hyperref}
\hypersetup{hidelinks, colorlinks}
\hypersetup{
    citecolor = {teal},
    linkcolor = {black},
}
\usepackage[capitalise]{cleveref}

\theoremstyle{plain}
\newtheorem{thm}{Theorem}[section]
\newtheorem{cor}[thm]{Corollary}
\newtheorem{prop}[thm]{Proposition}
\newtheorem{lem}[thm]{Lemma}

\theoremstyle{definition}
\newtheorem{defn}[thm]{Definition}
\newtheorem{exmp}[thm]{Example}

\newtheorem{assm}[thm]{Assumption}

\theoremstyle{remark}
\newtheorem{rem}[thm]{Remark}
\AtBeginEnvironment{rem}{%
    \small
  \pushQED{\qed}%
}
\AtEndEnvironment{rem}{\popQED}

\numberwithin{equation}{section}

\calclayout

\AtBeginDocument{%
   \def\MR#1{}
}

\newcommand{\loopeq}{t}
\newcommand{\teq}{\lambda}
\newcommand{\sg}{{\mathbb{Z}/p}}
\newcommand{\tg}{T}
\newcommand{\T}{\mathbf{T}}

\newcommand{\loc}{{\mathrm{loc}}}
\newcommand{\seq}{h}
\newcommand{\zpeq}{{[\![t, \theta]\!]}}
\newcommand{\cochar}{\beta}

\newcommand{\cupprod}{\cup}

\title[Quantum Steenrod operations of symplectic resolutions]{Quantum Steenrod operations of symplectic resolutions}

\author{Jae Hee Lee}
\email{jaehelee@mit.edu}

\begin{document}

\begin{abstract}
    We study the mod $p$ equivariant quantum cohomology of conical symplectic resolutions. We conjecture that the quantum Steenrod operations on divisor classes agree with the $p$-curvature of the mod $p$ equivariant quantum connection, and verify this in the case of the Springer resolution. The key ingredient is a new compatibility relation between the quantum Steenrod operations and the shift operators.
\end{abstract}

\maketitle
\section{Background and Results}\label{sec:intro}

\subsection{A brief overview of context}\label{ssec:intro-context}
The quantum connection of a symplectic manifold $(X, \omega)$ is a flat connection acting on singular cohomology $H^*(X)$, defined out of $3$-pointed genus $0$ Gromov--Witten invariants. A remarkable program from geometric representation theory, initiated in \cite{BMO11}, reveals that many flat connections arising from representation theory can be realized as equivariant versions of quantum connections of \emph{(conical) symplectic resolutions}; see \cite{BMO11, MO19, Dan20}.

The quantum connection of symplectic resolutions satisfy special properties, among which is that in  a suitable basis of cohomology, they could be defined over rings admitting mod $p$ reduction. That is, one can consider the quantum connection defined over $\mathbb{F}_p$. Connections in positive characteristic behave very differently from their characteristic zero counterparts. For example, there is a fundamental invariant called the \emph{$p$-curvature} which obstructs formal solvability, which is the subject of the famous Grothendieck--Katz conjecture \cite{Kat70, Kat72}.

Recently in representation theory, there has been important progress in understanding the mod $p$ (and $p$-adic) flat connections and their $p$-curvature \cite{SV19, SV-ar, Var21, EV23b}. Our goal is to apply techniques from symplectic Gromov--Witten theory to study these connections from the viewpoint of enumerative geometry.

Namely, we study the class of equivariant operations defined on mod $p$ quantum cohomology known as the \emph{quantum Steenrod operations} \cite{Fuk97, Wil20}. An important property of these operations is that they are covariantly constant with respect to the mod $p$ quantum connection \cite{SW22}. Despite many success with applications \cite{She21, CGG22, Sei19}, computations and more structural results concerning the quantum Steenrod operations beyond covariant constancy remained elusive.
 
In this paper, we prove an unexpected identification of the quantum Steenrod operations with the $p$-curvature of the quantum connections of many symplectic resolutions. The result follows from a new compatibility relation between the quantum Steenrod operations and \emph{shift operators} \cite{OP10, Iri17, LJ21}. Our result in particular expands the scope of available computations of quantum Steenrod operations to include the classical case of all Springer resolutions.

\subsection{Main results}\label{ssec:intro-results}
Fix an odd prime $p>2$, and take a cohomology class $b \in H^*(X;\mathbb{F}_p)$. The quantum Steenrod operations \cite{Fuk97} \cite{Wil20} are defined from $\mathbb{Z}/p$-equivariant counts of (perturbed) holomorphic curves $u: \mathbb{P}^1 \to X$ from $\mathbb{P}^1$ with $p$ marked points along the roots of unity, all incident to (Poincar\'e dual of) $b$. Following \cite{Lee23}, we generalize the quantum Steenrod operations to a setting where the target symplectic manifold $(X, \omega)$ admits Hamiltonian torus actions by $\T = (S^1)^r$. For a fixed (equivariant) cohomology class $b \in H^*_\T(X;\mathbb{F}_p)$ we define a \emph{$\T$-equivariant quantum Steenrod operation} (\cref{defn:eq-qst}) which is an endomorphism
    \begin{equation}
        \Sigma_b^\T : H^*_\T(X;\Lambda)\zpeq \to H^*_\T(X;\Lambda)\zpeq
    \end{equation}
of the underlying module of the equivariant quantum cohomology; here $\Lambda$ is the Novikov ring and $t, \theta$ represent the $\mathbb{Z}/p$-equivariant parameters $H^*_{\mathbb{Z}/p}(\mathrm{pt};\mathbb{F}_p) \cong \mathbb{F}_p\zpeq$ of degrees $|t|=2, |\theta|=1$ (since $p>2$, note that $\theta^2 = 0$).

In the presence of a torus $\T$ acting on $X$, there is another set of operations on equivariant quantum cohomology known as the \emph{shift operators} (\cref{defn:shift}) for each choice of a cocharacter $\cochar: S^1 \to \T$, 
\begin{equation}
    \mathbb{S}_\cochar : H^*_{\T} (X;\Lambda) \zpeq \to H^*_{\T} (X;\Lambda) \zpeq.
\end{equation}

The key property of the shift operator is that it does not give linear map of $H^*_{\T}(\mathrm{pt})\zpeq$-modules, but rather satisfies a ``twisted-linear'' condition \eqref{eqn:shifted-linear}. Following \cite{Iri17} and \cite{LJ20}, we define the shift operator as a composition of a linear map $S_\cochar$ (the ``bare'' shift operators) and a reparametrization map accounting for the failure of linearity (and hence the shift), see \cref{defn:shift}.

Our first main result is the following new compatibility relation of the (equivariant) quantum Steenrod operations $\Sigma_b^\T$ with the shift operators $\mathbb{S}_\cochar$.
\begin{thm}[{\cref{thm:compatibility}}]\label{thm:intro-compatibility}
    The quantum Steenrod operations $\Sigma_b^\T$ and the shift operators $\mathbb{S}_\cochar$ commute. Equivalently, the bare shift operators $S_\beta$ commute with $\Sigma_b^\T$ up to a shift in the equivariant parameters:
        \begin{equation}
        \Sigma_b^\T |_{\lambda \mapsto \lambda - \cochar t} \circ S_\cochar = S_\cochar \circ \Sigma_b^\T.
    \end{equation}
\end{thm}

\cref{thm:intro-compatibility} is a new constraint for the quantum Steenrod operations. A surprising consequence of this compatibility relation is that it constrains the quantum Steenrod operation enough so that it has to agree with the $p$-curvature of the mod $p$ quantum connection in many cases.

The $p$-curvature is a fundamental invariant of any connection in positive characteristic, and for the mod $p$ quantum connection it takes the following form (\cref{defn:p-curvature}) for a choice of $b \in H^2(X;\mathbb{F}_p)$:
\begin{equation}
        F_b = \nabla_b ^p - t^{p-1} \nabla_b : H^*_\T(X;\Lambda)\zpeq \to H^*_\T(X;\Lambda)\zpeq.
\end{equation}
Remarkably, it can be shown that the $p$-curvature of the equivariant quantum connections for conical symplectic resolutions $X$ satisfy the same properties as the quantum Steenrod operations observed in \cite{Lee23}, including covariant constancy (cf. \cref{thm:eq-cov-constancy}), compatibility with shift operators (cf. \cref{thm:compatibility}), and annihilation of \emph{arithmetic flat sections} of \cite{Var-ar} (cf. \cite[Section 7]{Lee23}).

Our main result is that in a large number of examples, the quantum Steenrod operations (on cohomology classes of degree $2$) are indeed the same as $p$-curvature. For the following result, let $X = T^*(G/B)$, the cotangent bundle of the complete flag variety of a complex semisimple Lie group $G$. We consider $X$ with the action of $\T = T \times S^1$, where $T$ is the natural action of the maximal (compact) torus and $S^1$ is the rotation of cotangent fibers, which scales the canonical holomorphic Liouville form. The cotangent bundle $X = T^*(G/B)$ with its projection to the nilpotent cone, the \emph{Springer resolution}, is the most classical example of a conical symplectic resolution.

Under the hypothesis (\cref{assm:shift-ag-equals-sg}) that our analytically defined shift operators agree with the shift operators defined in the algebraic geometry literature, we prove the following.

\begin{thm}[{\cref{thm:qst-is-pcurv-nilpt}, \cref{cor:qst-is-pcurv}}]\label{thm:intro-pcurv}
     Consider the action of $\T = T \times S^1$ on $T^*(G/B)$, as described above. For almost all primes $p$, the $p$-curvature $F_b$ and the quantum Steenrod operation $\Sigma_b^\T$ for $b \in H^2(X;\mathbb{F}_p)$ agree.
\end{thm}

\begin{rem}\label{rem:what-about-not-H2}
The $p$-curvature endomorphism $F_b$ is only defined for classes $b \in H^2(X;\mathbb{F}_p)$ of degree $2$, but the quantum Steenrod operation $\Sigma_b^\T$ can be defined for classes of arbitrary cohomological degree. When the quantum cohomology algebra is generated by classes in $H^2$, \cref{thm:intro-pcurv} characterizes the quantum Steenrod operation for \emph{any} class in terms of iterations of the $p$-curvature via the quantum Cartan relation $\Sigma_b^\T \circ \Sigma_{b'}^\T = \pm \Sigma_{b \ast b'}^\T$, where $\ast$ denotes the quantum product.
\end{rem}

\begin{rem}\label{rem:why-springer}
    For the sake of concreteness, we have opted to highlight the example of the Springer resolution $T^*(G/B)$, for which all constructions are classical and can be made concrete. Nevertheless, we explain how to adapt the proof for a more general class of symplectic resolutions in the main body of the paper. 
    
    In particular, the same result with the same proof holds for e.g. generalized Springer resolutions $T^*(G/P)$ for parabolic $P \le G$, Nakajima quiver varieties of type $A$, hypertoric varieties, and Hilbert scheme of points in $\mathbb{C}^2$; in general it would hold under the assumptions that (i) $X^T$ is discrete and (ii) the quantum multiplication $b \ \ast$ has simple spectrum.
\end{rem}

\cref{thm:intro-pcurv} can be read in two ways. From the perspective of symplectic topology, the result provides \emph{full} computation of quantum Steenrod operations in all degrees for the class of symplectic resolutions we consider. The previously known computations of quantum Steenrod operations relied on covariant constancy \cite{SW22}, which alone cannot determine the operation completely in positive characteristic except in a few examples where additional degree arguments are available. In the case of $X = T^*{\mathbb{P}}^1$, we \cite{Lee23} computed the operation using direct geometric arguments, relying on the specific geometry of $T^*\mathbb{P}^1$ and previous results by \cite{Voi96} on the Aspinwall--Morrison multiple cover formula for Gromov--Witten invariants. \cref{thm:intro-pcurv} is a more systematic result on the computation, which also reveals a relationship of the operation with the $p$-curvature. From the perspective of representation theory, the result provides a direct enumerative geometry definition of the $p$-curvature. The $p$-curvature also plays a role in ``large center'' phenomena in positive characteristic, see \cref{ssec:intro-relatedwork} for more discussion.

\subsection{Methods}\label{ssec:intro-methods}
We give a brief discussion of the strategy of the proofs of \cref{thm:intro-compatibility} and \cref{thm:intro-pcurv}.

\cref{thm:intro-compatibility} follows a ``TQFT argument'' that identifies the composition of operators on both sides of the equation
\begin{equation}
        \Sigma_b^\T |_{\lambda \mapsto \lambda - \cochar t} \circ S_\cochar = S_\cochar \circ \Sigma_b^\T
\end{equation}
by expressing their structure constants, given by counts of certain $0$-dimensional moduli spaces, as counts of boundary points of a certain $1$-dimensional moduli space. The relevant moduli space counts the configuration of (perturbed) holomorphic sections $u: C \to E_\cochar$ of a Hamiltonian fibration $E_\cochar \to \mathbb{P}^1$ with fiber $X$ obtained by gluing two copies of $D^2 \times X$ along the boundary using the cocharacter $\cochar: S^1 \to T$. The source curve $C$ has $(p+2)$-marked points at $0$, $\infty$, and a $p$-tuple of marked points depending on a parameter $0 < r< \infty$ at $r\zeta, \cdots, r \zeta^p = r$ for $\zeta = e^{2\pi i/p}$. The limits as $r \to 0$ and $r \to \infty$ recover the two compositions. 

From \cref{thm:intro-compatibility}, we deduce \cref{thm:intro-pcurv}. The key observations are the following (\cref{prop:eq-qst-properties}, \cref{lem:pcurv-properties}) that

\begin{enumerate}[topsep=-2pt, label=(\roman*)]
    \item Both $F_b$ and $\Sigma_b$, under specialization of $\mathbb{Z}/p$-equivariant parameters $(t, \theta) \mapsto 0$, recover the $p$-fold power of $b \in H^2(X;\mathbb{F}_p)$ in quantum cohomology,
    \item Both $F_b$ and $\Sigma_b$, under specialization of Novikov (quantum) parameters $q^A \mapsto 0$, recover the classical Steenrod operation $\mathrm{St}(b) = b^p - t^{p-1} b \in H^*(X;\mathbb{F}_p)\zpeq$,
    \item Both $F_b$ and $\Sigma_b$ commute with the shift operators.
\end{enumerate}

As we expand the matrices of $F_b$ and $\Sigma_b$ as polynomials in the equivariant parameter $\seq$ for the $S^1$-action rotating the holomorphic symplectic form, (i) and (ii) imply that the leading asymptotics and the constant terms in these polynomials agree. This observation bounds the degrees of $\seq$ in the difference $F_b - \Sigma_b$. The key property (iii), for the cocharacter in the direction of the $S^1$-action scaling the holomorphic symplectic form, is then used to deduce that all eigenvalues of the difference $F_b - \Sigma_b$ must in fact vanish.

To show that such an expansion of $F_b$ and $\Sigma_b$ as polynomials of $\seq$ is well-defined, we work in a basis of equivariant cohomology provided by the theory of \emph{cohomological stable envelopes} of \cite{MO19}. For the case of $X = T^*(G/B)$, the stable envelopes are given by certain linear combinations of the classes represented by the conormals of Schubert varieties in $G/B$, see \cite{SZ20}. The main technical input is to show the compactness and transversality of the moduli spaces used to define the quantum Steenrod operations (\cref{prop:qst-regularity}) and shift operators (\cref{prop:qst-regularity}), when the input and output incidence constraints are given by the stable envelope classes. (See \cref{rem:ritter-zivanovic-regularity} that discusses the recent work of \cite{RZ23} which establishes a relevant maximum principle in more generality).

To conclude $F_b = \Sigma_b$, it remains to note that the regular semisimplicity of the quantum multiplication $b \ast$ implies generic regular semisimplicity of the $p$-curvature $F_b$. Since $[F_b, \Sigma_b] = 0$ by covariant constancy \cite{SW22}, together with \cref{thm:qst-is-pcurv-nilpt} we conclude the result.

\subsection{Related work}\label{ssec:intro-relatedwork}
The presented result develops the theme introduced in \cite{Lee23}, and is a part of a more general program of studying the symplectic topology and Floer theory of symplectic resolutions; see also recent foundational work by Ritter--\v Zivanovi\'c \cite{RZ23}. The case of conical symplectic resolutions provide a very rich class of examples to study from the symplectic point of view, which (i) are amenable to computation due to the presence of extra structures (such as the shift operators) and (ii) naturally generalize constructions from geometric representation theory. Indeed our result \cref{thm:intro-pcurv} provides an enumerative geometric interpretation of the $p$-curvature of the flat connections associated to the symplectic resolutions.

The key tool to our result are shift operators. The shift operators were introduced in \cite{OP10} to define a set of operators compatible with the quantum connection. In the context of equivariant mirror symmetry \cite{BMO11} and \cite{MO19}, the shift operators collectively form a compatible \emph{difference connection}, see also \cite{LJ21}. The shift operators play an important role in the study of the monodromy of the equivariant quantum connection \cite{BMO11} and toric mirror symmetry \cite{Iri17}. A symplectic geometric definition was provided in the thesis of Liebenschutz-Jones \cite{LJ20}, \cite{LJ21} and versions for actions of non-abelian Lie groups were defined and used in \cite{GMP23}.

The main result of \cite{SW22} was covariant constancy, which establishes the compatibility of the quantum Steenrod operation $\Sigma_b$ with the quantum \emph{differential} connection. In light of the above discussion, our main result \cref{thm:intro-compatibility} can be understood as a companion theorem to \cite{SW22}, as it establishes the compatibility of $\Sigma_b$ with the \emph{difference}  connection. There are very few operations which are compatible with both the quantum differential connection and the difference connection from the shift operators, and this constrains the quantum Steenrod operations enough to be equivalent to $p$-curvature.

The $p$-curvature is a fundamental invariant of a connection in positive characteristic, and is an object of classical interest, notably featuring in the Grothendieck--Katz $p$-curvature conjecture \cite{Kat72}. In the recent important work of Etingof--Varchenko \cite{EV23a, EV23b}, they develop the theory of \emph{periodic pencils of flat connections} and prove systematic results about the spectrum of $p$-curvature endomorphisms. Our equivariant quantum connections of conical symplectic resolutions are examples of periodic pencils of flat connections, where periodicity arises from the compatibility with the shift operators. For the main result of this paper, the precise results from the theory of \cite{EV23a, EV23b} can actually be bypassed, but we expect their theory to be a key ingredient for the further study of equivariant quantum connections in positive characteristic.

Finally, we speculate a potential role of the quantum Steenrod operations in the context of 3D mirror symmetry, through the \emph{quantum Hikita conjecture} \cite{KMP21}. 3D mirror symmetry posits that conical symplectic resolutions come in pairs of \emph{Coulomb branches} and \emph{Higgs branches}, and the quantum Hikita proposal identifies the action of the quantum connection of the Higgs branch with the $D$-module of twisted traces for the Coulomb branch. In positive characteristic, the Coulomb branch admits a power operation (a Frobenius-constant quantization) structure \cite{Lon21}, whose image in the $D$-module of twisted traces is preserved by the action of Steenrod operators. We expect that the corresponding image in the quantum $D$-module is spanned by the quantum Steenrod powers $\mathrm{QSt}(b) = \Sigma_b(1)$, preserved by the action of $p$-curvature. Such “large center” versions of the quantum Hikita conjecture have $K$-theoretic analogues, in which power operations in quantum $K$-theory \cite{Oko17, AMS23} analogous to quantum Steenrod operations should play a role.

\subsection{Organization of the paper}\label{ssec:intro-organizations}
In Section 2 we introduce quantum Steenrod operations, their torus-equivariant generalizations, and the shift operators. The key result in this section is \cref{thm:compatibility}, the compatibility of the quantum Steenrod operations with the shift operators.

In Section 3 we focus on the example of the Springer resolution $X = T^*(G/B)$ and show how the necessary compactness and transversality properties for the moduli spaces relevant for the definition of quantum Steenrod operations and shift operators can be established.

In Section 4 we review the specific properties of the equivariant quantum connections of symplectic resolutions and their $p$-curvature. We also briefly introduce the theory of periodic pencils of flat connections of Etingof--Varchenko \cite{EV23a}, in particular their structural result about the eigenvalues of the $p$-curvature for such connections.

In Section 5 we prove the main theorem which establishes the equivalence of the $p$-curvature of (equivariant) quantum connections with the quantum Steenrod operations on divisor classes, verified for $X = T^*(G/B)$. We explain the key properties we use for $T^*(G/B)$ and how the method would apply to a larger class of examples.

\subsection{Acknowledgements}\label{ssec:intro-acknowledgements}
We would like thank Paul Seidel for his continued encouragement. We thank Shaoyun Bai, Zihong Chen, Hunter Dinkins, Pavel Etingof, Vasily Krylov, Davesh Maulik, Dan Pomerleano, Ben Webster, Alexander Varchenko, and Nicholas Wilkins for stimulating discussions at various stages. We especially thank Pavel Etingof for explaining his ongoing work with Alexander Varchenko and suggesting its crucial relevance with the current project, leading to the main result.  Finally, we thank the anonymous referees for extremely helpful comments and suggestions that improved the exposition.

This work was partially supported by MIT Landis fellowship, by the
National Science Foundation through grant DMS-1904997, and by the Simons Foundation through
grants 6552299 (the Simons Collaboration on Homological Mirror Symmetry) and 256290 (Simons
Investigator).

\section{Quantum Steenrod operations and shift operators}\label{sec:operators}
In this section, we review the definitions and properties of the endomorphisms of equivariant quantum cohomology considered in the paper: the quantum Steenrod operations and the shift operators. The main result proved in the section is \cref{thm:compatibility}, which establishes a compatibility relation between the quantum Steenrod operations and the shift operators.

\subsection{Assumptions and notations}\label{ssec:operators-notation}

\subsubsection{Target geometry and Novikov rings}\label{sssec:operators-notation-target}
In this article, we will consider connected symplectic manifold $(X, \omega)$ with a compatible almost complex structure $J$, satisfying the spherical nonnegativity assumption:
\begin{assm}\label{assm:nonnegativity}
    There exists $\teq \ge 0$ such that for all $A \in \mathrm{im}(\pi_2(X) \to H_2(X;\mathbb{Z}))$, 
    \begin{equation}
        c_1(A) : = \langle c_1(TX, J), A \rangle = \teq \cdot \langle [\omega], A \rangle.
    \end{equation}
\end{assm}
In practice, our examples will be K\"ahler manifolds which are weakly Calabi--Yau in the sense that $c_1(X) = 0$. 

Associated to $(X, \omega)$ is the Novikov ring $\Lambda$, which is a graded $\mathbb{F}_p$-algebra generated by exponential symbols of form $q^A$. The symbols $q^A$ are indexed over $A \in \mathrm{im}(\pi_2(X) \to H_2(X;\mathbb{Z}))$ in the image of the Hurewicz homomorphism such that $A = 0$ or $\int_A \omega > 0$. For a fixed ring $R$ we define the Novikov ring to be
\begin{equation}
    \Lambda := \Lambda_{R} = \left\{ \sum_A c_A q^A : c_A \in R \right\}, \quad |q^A| = 2c_1(A),
\end{equation}
where given a bound on $\int_A \omega$, only finitely many coefficients $c_A$ can be nonzero for $A$ satisfying the bound. Hence an element in $\Lambda$ can involve an infinite sum of $q^A$, and will do so for most examples considered in the current article. We also assume $R = \mathbb{F}_p$.

\subsubsection{Source curve}\label{sssec:operators-notation-source}
Denote $C =\mathbb{P}^1 = \mathbb{C} \cup \{ \infty \}$ to be a curve with distinguished marked points at
\begin{equation}
z_0 = 0, \ z_1 = \zeta = e^{2\pi i / p}, \dots, z_p = \zeta^p = 1,\  z_\infty = \infty.
\end{equation}
The curve $C$ carries its standard almost complex structure $j_C$, and will be taken as the source of the (perturbed) $J$-holomorphic maps. Let $\tau$ be the rotation of the curve $C$, which cyclically permutes the marked points as
\begin{equation}                                  
    \tau : (C, z_0, z_1 , z_2 \dots, z_p, z_\infty) \mapsto (C, z_0, z_2, z_3, \dots, z_1, z_\infty).
\end{equation}
The rotation $\tau$ defines an action of the cyclic group $\sg$ on $C$. We choose and fix an isomorphism $\mathbb{Z}/p$ with the multiplicative group of $p$th roots of unity, where the generator is sent to $\zeta = e^{2\pi i/p}$. For $\eta = \zeta^j \in \mathbb{Z}/p$, we denote by $\tau(\eta) = \tau^j$ the corresponding rotation of $C$.

\subsubsection{Equivariant cohomology}\label{sssec:operators-notation-eqcoh}
Let $G$ be a compact Lie group, which is taken to be $\mathbb{Z}/p$ or a torus $T = (S^1)^r$ in the current article. Define equivariant cohomology $H^*_G(X) := H^*(EG \times_G X)$ via the Borel construction. For chain-level constructions, we fix a choice for the contractible space $EG$ with a free (right) $G$-action and its quotient, the classifying space $BG$.

Constructions of quantum Steenrod operations involve the rotational symmetry group $\sg$ of the source curve $C$. The model for the classifying space $B\sg$ is as follows, see \cite[Section 2.1]{Lee23} for a more detailed description.

Fix a prime $p>2$ and the cyclic group $\sg$. Take
\begin{equation}
    S^\infty = \{ w \in \mathbb{C}^\infty = \bigoplus_{i=0}^\infty \mathbb{C} : \|w\|^2 = 1\}
\end{equation}
again as the model of $E\sg$, where $\eta \in \sg$ acts freely by multiplication of $\zeta = e^{2\pi i / p}$. Denote this action of $\sg$ by $\sigma$. The space $E\sg = S^\infty$ can be equipped with a $\sg$-CW complex structure with exactly $p$ many $i$-cells $\Delta^i, \sigma \Delta^i, \dots, \sigma^{p-1} \Delta^i$ in each dimension $i \ge 0$. The $i$-cell $\Delta^i$ has a natural compactification with stratified boundary, whose top strata in the boundary are given by the union $\bigsqcup_{j=0}^{p-1} \sigma^j \Delta^{i-1}$ of $(i-1)$-cells for $i$ even and $\sigma \Delta^{i-1} - \Delta^{i-1}$ for $i$ odd. The images of $\Delta^i$ under the quotient $E\sg \to B\sg$ induces a CW-structure on $B\sg$ with unique $i$-cell for every $i \ge 0$. For $\mathbb{F}_p$-coefficients, the images of (the compactification of) $\Delta^i$ in $B\sg$ become cycles, additively generating $H_*^\sg(\mathrm{pt};\mathbb{F}_p) = H_*(B\sg;\mathbb{F}_p)$.

Fix generators $\theta \in H^1(B\sg; \mathbb{F}_p)$ and $\loopeq \in H^2(B\sg;\mathbb{F}_p)$ such that
\begin{equation}
    \langle \theta, \Delta^1 \rangle =1, \quad \langle t, \Delta^2 \rangle = 1.
\end{equation}
Then $t, \theta$ generate $H^*_{\sg}(\mathrm{pt}, \mathbb{F}_p) \cong \mathbb{F}_p [\![t,\theta]\!]$. Our notation $\mathbb{F}_p [\![t,\theta]\!]$ is for graded (or degreewise) completion of the graded algebra $\mathbb{F}_p[t, \theta]$ by the ideal generated by $t, \theta$; in particular $\mathbb{F}_p[\![t, \theta]\!]$ is isomorphic as graded algebra to $\mathbb{F}_p[t, \theta]$.

Additively, the $\sg$-equivariant cohomology of a point has one generator in each degree $i \ge 0$, which we will denote by
\begin{equation}
    (t, \theta)^i = \begin{cases} t^{i/2} & i \mbox{ even } \\ t^{(i-1)/2}\theta & i \mbox{ odd } \end{cases}.
\end{equation}

In this article, we consider target symplectic manifolds $X$ with a Hamiltonian action of a torus $T = (S^1)^{r}$. Operations we consider will be defined on the Borel $T$-equivariant cohomology of $X$.

We again take $S^\infty$ with the multiplication action of $S^1 \le \mathbb{C}^*$ as a model of $ES^1$, so that $BS^1 = \mathbb{CP}^\infty$ has one $i$-cell for each $i \in 2 \mathbb{Z}_{\ge 0}$. Correspondingly, $H^*_{S^1}(\mathrm{pt};\mathbb{F}_p) = \mathbb{F}_p [\![\teq_1]\!]$ for a choice of the generator $\teq_1 \in H^2(\mathbb{CP}^\infty;\mathbb{F}_p)$, given by the first Chern class of the universal line bundle $\mathcal{O}_{\mathbb{CP}^\infty}(1)$. For a higher-dimensional torus $T = (S^1)^r$, take the product $ET = (S^\infty)^r$ with the product action so that $H^*_T(\mathrm{pt};\mathbb{F}_p) = \mathbb{F}_p [\![ \teq_1, \dots, \teq_r]\!].$ When the torus $T$ is fixed, we will denote this as $ H^*_T(\mathrm{pt};\mathbb{F}_p) := \mathbb{F}_p[\![  {\teq}]\!]$, where the variable $\teq$ is understood as a shorthand for the tuple of parameters $(\lambda_1, \dots, \lambda_r)$.

Inverting the equivariant parameters $\teq_i$ gives rise to the localized equivariant cohomology,
\begin{equation}\label{eqn:loc-cohomology}
    H^*_T(X;\mathbb{F}_p)_{\mathrm{loc}} := H^*_T(X;\mathbb{F}_p) \otimes_{H^*_T(\mathrm{pt};\mathbb{F}_p)} \mathrm{Frac}(H^*_T(\mathrm{pt};\mathbb{F}_p)).
\end{equation}

Note that locally finite cycles in $X$ represent cohomology classes in $H^*(X;\mathbb{F}_p)$ by Poincar\'e duality. The $T$-invariant locally finite cycles in $X$ give rise to a locally finite cycle in the Borel construction $X \times_T ET \to BT$ that is fiberwise (i.e. in $X$) the given locally finite cycle, and correspondingly represent cohomology classes in $H^*_T(X;\mathbb{F}_p)$. For the running example of $X=T^*(G/B)$, it is classical that its cohomology over $\mathbb{Z}$ is additively generated by locally finite cycles given by the unions of conormals to Schubert subvarieties of $G/B$ \cite{Bri05}, see \cref{sec:moduli-regular}.

\subsubsection{Equivariant quantum cohomology}\label{sssec:operators-notation-QH}
For closed $X$, the ordinary cohomology of $X$ admits the nondegenerate Poincar\'e pairing $(\cdot, \cdot) : H^*(X;\mathbb{F}_p) \otimes H^*(X;\mathbb{F}_p) \to \mathbb{F}_p$ given by $(b, b') = \int_X b \cupprod b'$. After fixing a basis of $H^*(X)$, for a fixed basis element $b \in H^*(X;\mathbb{F}_p)$, we denote its linear dual under the Poincar\'e pairing by $b^\vee \in H^*(X;\mathbb{F}_p)$. If $X$ admits an action of $T$, the Poincar\'e pairing can be extended to the equivariant cohomology ring $(\cdot, \cdot)_T : H^*_T(X) \otimes H^*_T(X) \to H^*_T(\mathrm{pt})$ using the pushforward under the proper map $X \times_T ET \to BT$.

If $X$ is non-compact, the Poincar\'e pairing becomes a pairing between cohomology and compactly supported cohomology, $(\cdot, \cdot) : H^*(X;\mathbb{F}_p) \otimes H^*_c(X;\mathbb{F}_p) \to \mathbb{F}_p$, see \cite[Section 1]{Bri00}. If $X$ is non-compact but $X^T$ is compact, the Poincar\'e pairing can still be defined on the localized equivariant cohomology ring $H^*_T(X)_{\mathrm{loc}}$ from \eqref{eqn:loc-cohomology} via localization without introducing compactly supported cohomology, see \cite[5.2]{AP93}.

We place the following assumption about the topology of the target manifold $X$, for simplicity; it is indeed satisfied for the examples considered in this paper.
\begin{assm}\label{assm:eq-formality}
    The singular cohomology $H^*(X;\mathbb{F}_p)$ is concentrated in even degrees, and the nontrivial cohomology classes $b \in H^*(X;\mathbb{F}_p)$ can be represented as Poincar\'e duals of (locally finite) cycles in $X$ given by unions of embedded submanifolds. Moreover, assume that $X^T$ is nonempty and compact.
\end{assm}
\begin{rem}\label{rem:eq-formality}
    By the Serre spectral sequence attached to the Borel construction $X \times_T ET \to BT$, the assumption here implies \emph{equivariant formality}, namely that $H^*_T(X)$ is freely generated by the non-equivariant cohomology $H^*(X)$ as a $H^*_T(\mathrm{pt}) = H^*(BT)$-module. In particular, $H^*_T(\mathrm{pt})$-linear operations on $H^*_T(X)$ may be defined on $H^*(X)$ and extended linearly.
\end{rem}

Let $X$ be a symplectic manifold for which its three-pointed ($T$-equivariant) genus $0$ Gromov--Witten invariants are defined over $\mathbb{Z}$; this includes the class of closed semipositive symplectic manifolds \cite[Section 6.4]{MS12} and the non-compact examples we will consider in this article \cite{BMO11, MO19}.

We choose a basis of $H^*(X;\mathbb{F}_p)$ denoted by $\{b \in H^*(X)\}$, which also generate $H^*_T(X)$ as a $H^*_T(\mathrm{pt})$-module by \cref{assm:eq-formality}.

\begin{defn}\label{defn:eq-quantum-product}
    The \emph{equivariant quantum product} $\ast = \ast_T$ is defined by means of the Poincar\'e pairing for $b_1, b_2 \in H^*_T(X;\mathbb{F}_p)$ as
    \begin{equation}
        b_1 \ast_T b_2 = \sum_{b} \sum_{A} \sum_{i_1, \dots, i_r \ge 0} \langle b_1, b_2, b \rangle_{A}^i \  {\teq}^{i} \ q^A \ b^\vee \in H^*_T(X;\Lambda),
    \end{equation}
    where $\langle b_1, b_2, b \rangle_A^i \  {\teq}^i = \langle b_1, b_2, b \rangle_A^i \ \teq_1^{i_1} \cdots \teq_r^{i_r} \in H^*_T(\mathrm{pt};\mathbb{F}_p)$ denotes the (mod $p$ reduction of) three-pointed $T$-equivariant genus $0$ Gromov--Witten invariants of degree $A$ with incidence constraints given by Poincar\'e dual cycles of $b_1, b_2, b \in H^*_T(X;\mathbb{F}_p)$. The first sum is over the chosen basis elements of $H^*(X;\mathbb{F}_p)$. The total sum is well-defined as an element of $H^*_T(X;\Lambda)$ by Gromov compactness. Extending the product linearly for the Novikov parameters $q^A \in \Lambda$, we obtain a product on $H^*_T(X;\Lambda)$ which is associative and commutative.
\end{defn}

By extending linearly for the equivariant parameters $t, \theta \in H^*_{\mathbb{Z}/p}(\mathrm{pt};\mathbb{F}_p)$, one can also extend the product to $H^*_T(X;\Lambda) [\![t,\theta]\!]$. 

\subsection{Quantum Steenrod operations}\label{ssec:operators-qst}
In this subsection, we review the definition of quantum Steenrod operations from \cite{Fuk97}, \cite{Wil20}, in the form introduced in \cite{SW22} and \cite{Lee23}. 

In the presence of a Hamiltonian action of a torus $T \cong (S^1)^r$ on the target symplectic manifold, we will later define and consider the $T$-equivariant generalizations of these operations. 

The structure constants of the quantum Steenrod operations are given by counting solutions $u: C \to X$ to perturbed Cauchy--Riemann equations with incidence constraints for the distinguished source curve from \cref{ssec:operators-notation}.

Following \cite{Lee23}, we implement the incidence constraints via genuine cycles. We again fix a basis of $H^*(X;\mathbb{F}_p)$ whose elements are denoted $b_\bullet$ for some subscript $\bullet$. Fix a basis cohomology class $b = \mathrm{PD}[Y] \in H^*(X;\mathbb{F}_p)$ that is (a mod $p$ reduction of) Poincar\'e dual to a locally finite homology cycle $[Y]$ represented by a union $Y = \bigcup Y_{(i)}$ of embedded submanifolds $Y_{(i)} \subseteq X$ (using \cref{assm:eq-formality}). Denote the map $Y \to X$ determined by the inclusions $Y_{(i)} \subseteq X$ as $Y \subseteq X$. Also fix $b_0 \in H^*(X)$, $b_\infty \in H^*_c(X)$ Poincar\'e dual to a locally finite cycle $[Y_0]$ and (compact) cycle $[Y_\infty]$, respectively. Note that $b_\infty^\vee \in H^*(X)$ from the Poincar\'e pairing. Let $Y_1 = \cdots = Y_p = Y$ be copies of the inclusion $Y \to X$.

Denote by $\sigma_{X}$ the action of $\mathbb{Z}/p$ which cyclically permutes the product $X^p$:
\begin{align}\label{eqn:X^p-cyclic-permutation}
    \sigma_{X} :& \ \mathbb{Z}/p \times  X \times X^p \times X \to X \times X^p \times X, \\
    & [\eta; (x_0; x_1, x_2, \dots, x_{p}; x_\infty)] \mapsto (x_0; x_{1-\eta}, x_{2-\eta}, \dots, x_{p-\eta}; x_\infty).
\end{align}

\begin{defn}\label{defn:eq-incidence-cycle}
    An \emph{incidence cycle} is the distinguished map $\mathcal{Y} = Y_0 \times (Y \times \cdots \times Y) \times Y_\infty \subseteq X \times X^p \times X$ or its image. An \emph{equivariant incidence cycle} corresponding to $\mathcal{Y}$ is a smooth map
    \begin{equation}
        \mathcal{Y}^{eq} : E\mathbb{Z}/p \times \mathcal{Y} \to X \times (E\mathbb{Z}/p \times X^p) \times X
    \end{equation}
    such that
    \begin{enumerate}[topsep=-2pt, label=(\roman*)]
        \item For each $w \in E\mathbb{Z}/p$, the restriction $\mathcal{Y}_w := \mathcal{Y}^{eq}|_{\{w\} \times \mathcal{Y}}$ maps into $X \times (\{w\} \times X^p) \times X$, hence can be considered as a map $\mathcal{Y}_w : \mathcal{Y} \to X \times X^p \times X$;
        \item For each $w \in E \mathbb{Z}/p$, $\mathcal{Y}_w$ is a product of maps $Y_j \to X$ representing a cycle homologous to the distinguished inclusion $Y_j \subseteq X$; denote the image of this product map also by $\mathcal{Y}_w$;
        \item Understood as maps from $\mathcal{Y}$, we have $\mathcal{Y}_{\eta \cdot w} = \sigma_{X}(\eta)^{-1} \circ \mathcal{Y}_w$ for $\eta \in \mathbb{Z}/p$.
    \end{enumerate}
\end{defn}

Given an incidence cycle $\mathcal{Y}$, a corresponding equivariant incidence cycle $\mathcal{Y}^{eq}$ can be constructed inductively (in $i$) over the finite-dimensional cells $\Delta^i \subseteq E\mathbb{Z}/p$: we choose an arbitrary extension of the data defined over the $(i-1)$-skeleton of $E\mathbb{Z}/p$ to $\Delta^i \subseteq E\mathbb{Z}/p$ so that condition (ii) is satisfied, and equivariance condition (iii) fixes the data over all other $i$-cells in $E\mathbb{Z}/p$. Here we used $Y_1 = \cdots = Y_p = Y$. In particular, $\mathcal{Y}^{eq}$ is only unique up to ($\mathbb{Z}/p$-equivariant) homotopy.

For the distinguished $i$-cell $\Delta^i \subseteq E\mathbb{Z}/p$, denote by $\mathcal{Y}^{eq; i} := \mathcal{Y}^{eq}|_{\Delta^i \times \mathcal{Y}}$ the restriction of the equivariant incidence cycle to $\Delta^i \subseteq E\mathbb{Z}/p$.

For $X$, fix a compatible almost complex structure $J$. Over $C \times X$, there are complex vector bundles $TC$ and $TX$ from pullback. An equivariant inhomogeneous term is a choice of a $J$-complex antilinear bundle homomorphism
\begin{equation}
    \nu^{eq} \in C^\infty ( E\mathbb{Z}/p \times C \times X ; \mathrm{Hom}^{0,1}(TC, TX))
\end{equation}
that satisfies the equivariance condition
\begin{equation}
    \nu_{\eta \cdot w, z, x} = \nu_{w, \tau(\eta)(z), x} \circ D (\tau)_z : TC_z \to TX_x.
\end{equation}
Assume that $\nu^{eq}$ is supported away from the neighborhoods of distinguished marked points on $C$.

Denote by $\mathcal{M}^{eq}_A$ the moduli space of solutions to the parametrized problem
\begin{equation}
    \mathcal{M}^{eq}_A := \left\{ (u:C \to X , w \in E\mathbb{Z}/p)  :  u_*[C] =A, \ (\overline{\partial}_J u)_z = \nu_{w, z, u(z)} \right\}.
\end{equation}

The moduli space carries an evaluation map
\begin{align}
    \mathrm{ev}^{eq}: \  &\mathcal{M}^{eq}_A \to X \times (E\mathbb{Z}/p \times X^p) \times X \\
    &(u,w) \mapsto (u(z_0) ; (w, u(z_1), \dots, u(z_p)) ; u(z_\infty)).
\end{align}
Provided the boundedness of the moduli space $\mathcal{M}^{eq}_A$, by Gromov compactness the evaluation map is covered by parametrized versions of simple stable maps \cite[Chapter 6]{MS12}, where the parametrization is by $E\mathbb{Z}/p$.

Through the evaluation map, one can further cut down the moduli space using the incidence cycles:
\begin{equation}
    \mathcal{M}^{eq}_A \pitchfork \mathcal{Y}^{eq} := \left\{ (u,w) \in \mathcal{M}^{eq}_A : \  \mathrm{ev}^{eq}(u,w) \in \mathcal{Y}_w \right\}.
\end{equation}
\begin{assm}\label{assm:qst-regular}
    We fix $J$, $\nu^{eq}$ and a choice of $\mathcal{Y}^{eq}$ such that the following requirements hold:
       \begin{enumerate}[topsep=-2pt, label=(\roman*)]
        \item The moduli space $\mathcal{M}_A^{eq} \pitchfork \mathcal{Y}^{eq} $ is regular;
        \item The evaluation map from moduli space of simple stable maps (with at least one bubble) is transverse to $\mathcal{Y}^{eq}$. In particular, when the counts $\mathcal{M}_A^{eq} \pitchfork \mathcal{Y}^{eq} $ are rigid, the simple stable maps only intersect $\mathcal{Y}^{eq}$ along $\mathcal{M}_A^{eq}$.
    \end{enumerate}
\end{assm}
When the target symplectic manifold $X$ is compact, these requirements are satisfied for a generic choice of $J, \nu^{eq}$, and $\mathcal{Y}^{eq}$ by standard arguments, cf. \cite[Section 6.7]{MS12} and \cite[Section 2.3]{Lee23}. We will address the regularity issues for moduli spaces in the setting of the non-compact targets relevant for our computations in \cref{sec:moduli-regular}.

By restricting to $w \in \Delta^i \subseteq E\mathbb{Z}/p$, we obtain the $i$th equivariant moduli space
\begin{equation}
     \mathcal{M}^{eq; i}_A \pitchfork \mathcal{Y}^{eq; i} := \left\{ (u,w) \in \mathcal{M}^{eq}_A \pitchfork \mathcal{Y}^{eq} : w \in \Delta^i \right\}.
\end{equation}
of expected dimension $i + \dim_{\mathbb{R}}X +  2c_1(A) - |b_0| - p |b| - |b_\infty|$. These moduli spaces are compactified by the parametrized stable map constructions, and also by allowing the parameter $w \in \Delta^i$ to approach $\partial \Delta^i$, see \cite[Section 4]{SW22} and \cite[Section 2]{Lee23} for more details.

\begin{lem}[{\cite[Lemma 2.15]{Lee23}}]\label{lem:qst-Fp-count}
    Suppose \cref{assm:qst-regular} holds. Fix the unique $i \ge 0$ such that $\dim \mathcal{M}^{eq; i}_A \pitchfork \mathcal{Y}^{eq; i} = 0$. Then the mod $p$ count of this $0$-dimensional moduli space is independent of the choices of $J$, $\nu^{eq}$, and $\mathcal{Y}^{eq}$.
\end{lem}

\begin{rem}
We do not reproduce the full argument here, but the key observation for showing that the counts are only well-defined in $\mathbb{F}_p$ is as follows. In a $1$-parameter family of moduli spaces corresponding to a homotopy of defining data, there are contributions arising from the compactifying strata where $w \in \partial \Delta^i$, so the $1$-parameter family fails to yield a cobordism between the moduli at the two endpoints. Nevertheless, by $\mathbb{Z}/p$-equivariance, such contributions always occur in multiples of $p$ and the count of the moduli space is still well-defined modulo $p$.
\end{rem}

One can now define the quantum Steenrod operations using the counts from the $i$th equivariant moduli spaces as the structure constants. Recall that $(t,\theta)^i$ denotes the monomial of degree $i$ generating $H^i_\sg(\mathrm{pt};\mathbb{F}_p)$, given by $t^{i/2}$ or $t^{i/2}\theta$ depending on the parity of $i$.

\begin{defn}
    Fix $b \in H^*(X;\mathbb{F}_p)$. The \emph{quantum Steenrod operations} for $b$ is a map
    \begin{equation}
        \Sigma_b : H^*(X;\Lambda)\zpeq \to H^*(X;\Lambda)\zpeq
    \end{equation}
    defined for $b_0 \in H^*(X;\mathbb{F}_p)$ as the sum
    \begin{equation}
        \Sigma_b(b_0) = \sum_{b_\infty} \sum_A \sum_{i \ge 0} (-1)^\star \left( \#  \mathcal{M}^{eq; i}_A \pitchfork \mathcal{Y}^{eq; i} \right) q^A (t,\theta)^i\  b_\infty^\vee \in H^*(X;\Lambda)\zpeq
    \end{equation}
    and extended $q, t, \theta$-linearly. The first sum is over a chosen basis of $H^*(X;\mathbb{F}_p)$ as a vector space, and $b_\infty^\vee$ denotes the dual of $b_\infty$ under the Poincar\'e pairing.
\end{defn}
Here the sign is given by
\begin{equation}\label{eqn:qst-sign}
    \star = \begin{cases} |b||b_0| & i \mbox{ even} \\ |b||b_0|+|b|+|b_0| & i \mbox{ odd} \end{cases}
\end{equation}
and in particular can be ignored if the odd cohomology of $X$ vanishes. The counts of moduli spaces are defined to be zero if the moduli space is not discrete, i.e.
\begin{equation}\label{eqn:qst-dimension-constraint}
    \#  \mathcal{M}^{eq; i}_A \pitchfork \mathcal{Y}^{eq; i} := \begin{cases} \#  \mathcal{M}^{eq; i}_A \pitchfork \mathcal{Y}^{eq; i} \in \mathbb{F}_p & i + \dim_{\mathbb{R}}X +  2c_1(A) - |b_0| - p |b| - |b_\infty| = 0  \\ 0 & \mbox{otherwise} \end{cases}
\end{equation}
This dimension constraint implies that $\Sigma_b$ is a map of degree $p|b|$.

The key properties of quantum Steenrod operations are proved in \cite{Wil20} and \cite{SW22}, which we record here. Recall that $\mathrm{St}(b) \in H^*(X;\mathbb{F}_p)\zpeq$ is the total Steenrod power of $b \in H^*(X;\mathbb{F}_p)$ in singular cohomology; it is defined by the composition
\begin{align}\label{eqn:classical-steenrod}
    H^*(X;\mathbb{F}_p) &\to H^*_{\mathbb{Z}/p}(X^p ; \mathbb{F}_p) \overset{\Delta^*}{\longrightarrow} H^*_{\mathbb{Z}/p}(X;\mathbb{F}_p) \cong H^*(X;\mathbb{F}_p) \zpeq \\
    b & \mapsto b^{\otimes p} \mapsto \mathrm{St}(b),
\end{align}
where $\mathbb{Z}/p$ acts by cyclic permutation on $X^p$ and $\Delta : X \to X^p$ is the ($\mathbb{Z}/p$-equivariant) diagonal embedding.

\begin{prop}
    For $b \in H^*(X;\mathbb{F}_p)$, the operation $\Sigma_b$ satisfies the following properties.
    \begin{enumerate}[topsep=-2pt, label=(\roman*)]
        \item $\Sigma_1 = 1 \cupprod = \mathrm{id}$ for $1 \in H^0(X;\mathbb{F}_p)$, the unit class.
        \item $\Sigma_b(b_0)|_{(t,\theta) = 0} = \overbrace{b \ast \cdots \ast b}^{p} \ast \ b_0$, hence $\Sigma_b$ deforms the $p$th power in quantum cohomology in equivariant parameters.
        \item $\Sigma_b(b_0)|_{q^A = 0} = \mathrm{St}(b) \cupprod b_0$, hence $\Sigma_b$ deforms the classical Steenrod operations in quantum parameters.
    \end{enumerate}
\end{prop}

\begin{rem}\label{rem:st-vs-qst}
    The classical Steenrod operations in fact arise from the action of the symmetric group $S_p$ acting on $X^p$, not its subgroup $\mathbb{Z}/p \le S_p$ of cyclic permutations. Since $H^*_{S_p}(\mathrm{pt};\mathbb{F}_p) \cong \mathbb{F}_p[\![t^{p-1}, \theta]\!] \hookrightarrow H^*_{\mathbb{Z}/p}(\mathrm{pt}; \mathbb{F}_p) \cong \mathbb{F}_p \zpeq$, many coefficients of the expansion of $\mathrm{St}(b)$ from \eqref{eqn:classical-steenrod} in the equivariant parameters would in fact vanish. However, $\Sigma_b$ in general involves all monomials in $t, \theta$ that satisfy the dimension constraints \eqref{eqn:qst-dimension-constraint}.
\end{rem}

Another key property of the quantum Steenrod operation is its compatibility with the (small) quantum connection, proven in \cite{SW22}.

Fix an element $a \in H^2(X;\mathbb{Z})$. Define a Novikov differentiation operator $\partial_a : \Lambda \to \Lambda$ which acts on symbols $q^A$ by $\partial_a q^A = (a \cdot A)q^A$. Extend this linearly in the equivariant parameters to defined $\partial_a: \Lambda \zpeq \to \Lambda \zpeq$.

\begin{defn}
    The \emph{quantum connection} of $X$ is the collection of linear maps
    \begin{equation}
        \nabla_a : H^*(X;\Lambda)\zpeq \to H^*(X;\Lambda) \zpeq
    \end{equation}
    for each $a \in H^2(X;\mathbb{Z})$ defined by
    \begin{equation}
        \nabla_a \beta = t \partial_a \beta + a \ast \beta
    \end{equation}
    where $\ast = \ast_T|_{\teq = 0}$ is the (non-$T$-equivariant) quantum product.
\end{defn}

A key property of the quantum connection is that it is a \emph{flat connection}, i.e. different operators $\nabla_a$ and $\nabla_{a'}$ commute. The flatness of the connection is equivalent to the commutativity and associativity of the quantum product.

\begin{thm}[{\cite[Theorem 1.4]{SW22}}]\label{thm:cov-constancy-qst-noneq}
    For any $b \in H^*(X;\mathbb{F}_p)$, the map $\Sigma_b$ is covariantly constant for the quantum connection, that is it satisfies
    \begin{equation}
        [\nabla_a, \Sigma_b] = 0
    \end{equation}
    for any $a \in H^2(X;\mathbb{Z})$.
\end{thm}

This result and the observation that covariant constancy admits a $T$-equivariant generalization was the starting point of \cite{Lee23}.

\subsection{$T$-equivariant quantum Steenrod operations}\label{ssec:operators-qst-teq}
Let $X$ be a symplectic manifold equipped with an action of a Hamiltonian torus $T \cong (S^1)^r$, denoted $\rho: T \to \mathrm{Ham}(X,\omega)$. Denote the action by $\theta \cdot x = \rho(\theta)(x)$ for $\theta \in T$. In this subsection we introduce the generalizations of the quantum Steenrod operations in the $T$-equivariant situation. The special case of these generalizations for $T = S^1$ was considered in the article \cite{Lee23}.

To define the structure constants of the $T$-equivariant quantum Steenrod operations, we will follow a popular approach in symplectic topology (see e.g. \cite{SS10}, \cite{LJ21}), to use the ``Borel model'' to define the relevant moduli spaces. Namely, we allow the almost complex structure and inhomogeneous terms to vary over the parameter space $BT \cong (\mathbb{CP}^\infty)^r$; see \cref{sssec:operators-notation-eqcoh} for our choice of the model of $ET$ and $BT$.

\begin{defn}\label{defn:eq-J}
    A \emph{$T$-equivariant collection of almost complex structures} is the set $J^{eq} := \{ J_{w'} \}_{w' \in ET}$ where each $J_{w'} \in \mathrm{End}(TX)$ is a $\omega$-compatible almost complex structure smoothly depending on $w' \in ET = (S^\infty)^{r}$, satisfying the equivariance condition
    \begin{equation}
        D\rho \circ J_{\theta \cdot w', x} = J_{w', \rho(\theta)(x)} \circ D\rho : \ TX_x \to TX_{\theta \cdot x}.
    \end{equation}
\end{defn}

Since the space of $\omega$-compatible almost complex structures is contractible, one can inductively choose (for the finite dimensional approximations of $ET$) perturbations of a fixed $J$ to show the existence of $T$-equivariant collection of almost complex structures.

\begin{defn}\label{defn:eq-nu}
    A \emph{$(\mathbb{Z}/p \times T)$-equivariant inhomogeneous term} is the data of
    \begin{equation}
        \nu^{eq} = \{ \nu_{w, w', z, x} \in \mathrm{Hom}^{0,1}(TC,TX)\} \in C^\infty( (E\mathbb{Z}/p \times ET) \times C \times X; \mathrm{Hom}^{0,1}(TC, TX))
    \end{equation}
    satisfying the equivariance condition
    \begin{equation}
        \nu_{\eta \cdot w, \theta \cdot w', z, x} = D\rho(\theta)^{-1} \circ \nu_{w, w', \tau(\eta)(z), \rho(\theta)(x)} \circ D\tau(\eta) : TC_z \to TX_{x}.
    \end{equation}
\end{defn}

One can consider a $(\mathbb{Z}/p \times T)$-equivariant inhomogeneous term as a $BT$-parametrized version of the pervious construction of a ($\mathbb{Z}/p$-)equivariant inhomogeneous term, now valued in $J_{w'}$-complex antilinear homomorphisms. We use the same notation $\nu^{eq}$ as all discussion below will involve $(\sg \times T)$-equivariant inhomogeneous terms.

Similarly, one can define the following analogue of the incidence cycles from \cref{defn:eq-incidence-cycle}. Fix a cohomology class $b = \mathrm{PD}[Y] \in H^*_T(X;\mathbb{F}_p)$ that is Poincar\'e dual to a $T$-invariant locally finite homology cycle $[Y]$. Also fix $b_0 \in H^*_T(X)$, $b_\infty \in H^*_{c, T}(X)$ Poincar\'e dual to a $T$-invariant locally finite cycle $[Y_0]$ and $T$-invariant cycle $[Y_\infty]$, respectively. Again fix the distinguished map $\mathcal{Y} = Y_0 \times (Y \times \cdots \times Y) \times Y_\infty \subseteq X \times X^p \times X$; note that the image is $\mathbb{Z}/p \times T$-invariant.

\begin{defn}\label{defn:totaleq-incidence-cycle}
    An \emph{$(\sg \times T)$-equivariant incidence cycle} corresponding to $\mathcal{Y}$ is a smooth map
    \begin{equation}
        \mathcal{Y}^{eq} : E\mathbb{Z}/p \times ET \times \mathcal{Y} \to ET \times (X \times (E\mathbb{Z}/p \times X^p) \times X)
    \end{equation}
    such that
    \begin{enumerate}[topsep=-2pt, label=(\roman*)]
        \item For each $(w,w') \in E\mathbb{Z}/p \times ET$, the restriction $\mathcal{Y}_{w,w'} := \mathcal{Y}^{eq}|_{(w,w') \times \mathcal{Y}}$ maps into $\{w'\} \times X \times (\{w\} \times X^p) \times X$, hence can be considered as a map $\mathcal{Y}_{w,w'} : \mathcal{Y} \to X \times X^p \times X$;
        \item For each $(w,w') \in E \mathbb{Z}/p \times ET$, $\mathcal{Y}_{w,w'}$ is a product of maps $Y_j \to X$ representing a cycle homologous to the distinguished inclusion $Y_j \subseteq X$;
        \item Understood as maps from $\mathcal{Y}$, we have $\mathcal{Y}_{\eta \cdot w, \theta \cdot w'} = \rho(\theta)^{-1} \circ \sigma_{X}(\eta)^{-1} \circ \mathcal{Y}_{w,w'}$ where $(\eta, \theta) \in \sg \times T$, and $\rho$ acts diagonally on all factors of $X$.
    \end{enumerate}
\end{defn}

As the initial choice of $\mathcal{Y}$ is $(\mathbb{Z}/p \times T)$-invariant, we can construct the equivariant incidence cycle as a family of perturbations of $\mathcal{Y}$ parametrized over $E\sg \times ET$; see \cite[Lemma 5.5]{Lee23}.

Using the $(\mathbb{Z}/p \times T)$-equivariant almost complex structures, inhomogeneous terms and incidence cycles, one can now define the equivariant moduli spaces as
\begin{equation}
    \mathcal{M}^{eq}_A := \left\{ (u:C \to X , (w, w') \in E\mathbb{Z}/p \times ET)  :  u_*[C] =A, \ (\overline{\partial}_{J_{w'}} u)_z = \nu_{w, w', z, u(z)} \right\}
\end{equation}
which carries an evaluation map
\begin{align}
    \mathrm{ev}^{eq} : &\mathcal{M}_A^{eq} \to ET \times (X \times (E\mathbb{Z}/p \times X^p) \times X) \\
    &(u,w, w') \mapsto ( w'; u(z_0) ; (w, u(z_1), \dots, u(z_p)) ; u(z_\infty)).
\end{align}

Using the evaluation map one can also define
\begin{equation}
    \mathcal{M}^{eq}_A \pitchfork \mathcal{Y}^{eq} := \left\{ (u,w,w') \in \mathcal{M}^{eq}_A : \  \mathrm{ev}^{eq}(u,w,w') \in \mathcal{Y}_{w,w'} \right\}.
\end{equation}
Note that $(\eta, \theta) \in \sg \times T$ acts on $u: C \to X$ by
\begin{equation}
    (\eta, \theta) \cdot u = \rho(\theta) \circ u \circ \tau(\eta)^{-1}
\end{equation}
and hence $(u, \eta \cdot w, \theta \cdot w') \in \mathcal{M}^{eq}_A$ if and only if $((\eta, \theta) \cdot u, w, w') \in \mathcal{M}^{eq}_A$ by the equivariance condition on the inhomogeneous terms. The same holds for the version with incidence constraints, $\mathcal{M}^{eq}_A \pitchfork \mathcal{Y}^{eq}$.

One can restrict to $\Delta^i \times \Delta^{2j_1} \times \cdots \times \Delta^{2j_r} \subseteq E\sg \times ET$, where $\Delta^{2j_k}$ denotes the preimage of the unique $2j_k$-cell in the $k$th factor $BS^1 \subseteq BT$, to define the moduli spaces
\begin{equation}\label{eqn:i/2j-th-eq-moduli-space-qst}
    \mathcal{M}^{eq;i;2j}_A \pitchfork \mathcal{Y}^{eq;i;2j} := \left\{ (u,w,w') \in \mathcal{M}^{eq}_A  \pitchfork \mathcal{Y}^{eq} : (w, w') \in \Delta^i \times \Delta^{2j_1} \times \cdots \times \Delta^{2j_r} \right\}.
\end{equation}
The expected dimension of this moduli space is $i + \sum_{k=1}^r 2j_k + \dim_{\mathbb{R}} X + 2c_1(A) - |b_0| - p|b| - |b_\infty|$. As before, this moduli space are compactified using stable maps and degeneration of $(w, w') \in \Delta^i \times \prod \Delta^{2j_k}$ into their boundaries. Assume that the data $(J^{eq}, \nu^{eq}, \mathcal{Y}^{eq})$ are all chosen such that the counts $\# \mathcal{M}^{eq;i;2j}_A \pitchfork \mathcal{Y}^{eq;i;2j} \in \mathbb{F}_p$ are well-defined.

Let $H^*_T(X;\Lambda)$ denote the graded completed tensor product $H^*_T(X;\Lambda) := H^*_T(X;\mathbb{F}_p) \ \widehat{\otimes}_{\mathbb{F}_p} \Lambda$.

\begin{defn}\label{defn:eq-qst}
Fix $b \in H^*_T(X;\mathbb{F}_p)$. The \emph{$T$-equivariant quantum Steenrod operations} for $b$ is a map
    \begin{equation}
        \Sigma_b^T : H^*_T(X;\Lambda)\zpeq \to H^*_T(X;\Lambda)\zpeq
    \end{equation}
    defined for $b_0 \in H^*_T(X;\mathbb{F}_p)$ as the sum
    \begin{equation}
        \Sigma_b^T(b_0) = \sum_{b_\infty} \sum_A \sum_{i \ge 0, j \in (\mathbb{Z}_{\ge 0})^r} (-1)^\star \left( \#  \mathcal{M}^{eq; i;2j}_A \pitchfork \mathcal{Y}^{eq; i;2j} \right) q^A (t,\theta)^i\  {\teq}^{j} \  b_\infty^\vee \in H^*_T(X;\Lambda)\zpeq
    \end{equation}
    and extended $q, t, \theta, \teq$-linearly. The first sum is over a chosen basis of $H^*(X;\mathbb{F}_p)$ as a vector space, which also generate $H^*_T(X;\mathbb{F}_p)$ as a $H^*_T(\mathrm{pt};\mathbb{F}_p) = \mathbb{F}_p[\![  \teq]\!]$-module by \cref{assm:eq-formality}, which implies equivariant formality (\cref{rem:eq-formality}). 
\end{defn}
The sign $\star$ is as in \eqref{eqn:qst-sign}, but can be ignored under \cref{assm:eq-formality}. We drop the superscript $T$ and abbreviate $\Sigma_b := \Sigma_b^T$ in the following discussions when the group $T$ in concern is clear.

The following properties hold for the $T$-equivariant quantum Steenrod operations, which generalize the properties of the usual quantum Steenrod operations. Proof in the case $T = S^1$ is given in \cite[Proposition 5.12]{Lee23} and the higher rank case $T = (S^1)^r$ is identical.

\begin{prop}\label{prop:eq-qst-properties}
    For $b \in H^*_T(X;\mathbb{F}_p)$, the operation $\Sigma_b^T$ satisfies the following properties.
    \begin{enumerate}[topsep=-2pt, label=(\roman*)]
        \item $\Sigma_1^T = \mathrm{id}$ for $1 \in H^0_T(X;\mathbb{F}_p)$, the unit class.
        \item $\Sigma_b^T(b_0)|_{(t,\theta) = 0} = \overbrace{b \ast_T \cdots \ast_T b}^{p} \ast_T \ b_0$.
        \item $\Sigma_b^T(b_0)|_{\teq = 0} = \Sigma_b(b_0)$.
        \item $\Sigma_b^T(b_0)|_{q^A = 0} = \mathrm{St}^T(b) \cupprod_{T} b_0$, where $\mathrm{St}^T(b) \in H^*_{T}(X;\mathbb{F}_p)\zpeq$ is the class obtained by the power operation construction as in \eqref{eqn:classical-steenrod} in $T$-equivariant cohomology.
    \end{enumerate}
\end{prop}

Similarly, covariant constancy for the quantum connection generalizes to the $T$-equivariant setting.

\begin{defn}\label{defn:eq-qconn}
    The \emph{equivariant quantum connection} of $X$ is the collection of linear maps
    \begin{equation}
        \nabla_a^T : H^*_T(X;\Lambda)\zpeq \to H^*_T(X;\Lambda) \zpeq
    \end{equation}
    for each $a \in H^2_T(X;\mathbb{Z})$ defined by
    \begin{equation}
        \nabla_a^T \beta = t \partial_a \beta - a \ast_T \beta
    \end{equation}
    where $\ast_T$ is the $T$-equivariant quantum product, see \cref{defn:eq-quantum-product}. Here $\partial_a$ is the operator on $\Lambda$ which acts by $\partial_a q^A = \langle \bar{a}, A \rangle q^A$ where $\bar{a} \in H^2(X)$ is the image of $a$ under $H^2_T(X) \to H^2(X)$.
\end{defn}

\begin{thm}[{\cite[Theorem 5.14]{Lee23}}]\label{thm:eq-cov-constancy}
    For any $b \in H^*_T(X;\mathbb{F}_p)$, the map $\Sigma_b^T$ is covariantly constant for the equivariant quantum connection, that is it satisfies
    \begin{equation}
        [\nabla_a^T , \Sigma_b^T] = 0
    \end{equation}
    for any $a \in H^2_T(X;\mathbb{Z})$.
\end{thm}

\subsection{Shift operators}\label{ssec:operators-shift}
In this section we define the shift operators associated to cocharacters $\cochar : S^1 \to T$. The shift operators are extra operators defined in the presence of Hamiltonian $T$-action, acting on the underlying module of $T$-equivariant quantum cohomology $H^*_T(X;\Lambda)[\![t]\!]$. 

Following \cite{Iri17} and \cite{LJ21}, we define it as a composition of two homomorphisms: one geometrically obtained by section counting in a fibration associated to $\cochar$, and one algebraically obtained by reparametrizing the group action.

\subsubsection{Seidel fibration}
Fix a cocharacter $\beta : S^1 \to T$, whose restriction to $\sg \le S^1$ we also denote by $\cochar$. Associated to $\cochar$ is a Hamiltonian fibration $E_\cochar$ over $\mathbb{P}^1$ with fibers symplectomorphic to $(X, \omega)$, defined as follows. Take $S^3 = \{ (w_0, w_1) \in \mathbb{C}^2 : |w_0|^2 + |w_1|^2 = 1 \}$, the unit sphere in $\mathbb{C}^2$ with the diagonal action of $\theta \in S^1 \le \mathbb{C}^*$ by multiplication.
\begin{defn}
The \emph{Seidel fibration} associated to $\cochar: S^1 \to T$ is the space
\begin{equation}
    E_\cochar = (X \times S^3) / \left(x, (w_0, w_1)) \sim ( (\rho \circ \cochar)(\theta) x, (\theta w_0, \theta w_1) \right)
\end{equation}
with its projection $\pi_\cochar: E_\cochar \to S^3 / S^1 \cong \mathbb{P}^1$. 
\end{defn}
The etymology is based on a more general construction of such fibrations associated to a loop in Hamiltonian diffeomorphisms, which was introduced in \cite{Sei97}.

The key geometric property of $E_\cochar$ is that it can be given an action of $\sg \times T$ which restricts to different actions on the fibers over the two fixed points of the rotation action on the base. For a choice of $\cochar: S^1 \to T$, which we restrict to define $\cochar: \sg \to T$, consider following two actions of $\sg \times T$ on $X$.

Let $\rho_0, \rho_\cochar : \sg \times T \to \mathrm{Ham}(X, \omega)$ act so that
\begin{align}
    \rho_0(\eta, \theta) (x) &:= \rho(\theta) (x) = \theta \cdot x, \\
    \rho_\cochar(\eta, \theta) (x) &: = \rho(\theta - \beta(\eta)) (x) = (\theta - \cochar(\eta)) \cdot x.
\end{align}
Here we used an additive notation, by fixing an identification $T \cong (\mathbb{R}/\mathbb{Z})^{r}$.

Consider the following $(\sg \times T)$-action on $E_\cochar$, where $T$ acts fiberwise and $\sg$ acts by rotation of the base:
\begin{align}
    \rho_E( \eta, \theta):   E_\cochar &\to E_\cochar \\
     [ x, (a_0, a_1)] &\mapsto [\rho(\theta) x, (\eta a_0, a_1)].
\end{align}
With respect to this action, the inclusions of $0 = [0:1]$-fiber $\iota_0 : X \to E_\cochar$ and $\infty = [1:0]$-fiber $\iota_\infty : X \to E_\cochar$  respectively become equivariant maps $\iota_0 : (X, \rho_0) \to (E_\cochar, \rho_E)$ and $\iota_\infty: (X, \rho_\cochar) \to (E_\cochar, \rho_E)$. Denote by $X_0, X_\infty$ the embedded images of $X$ in $E_\cochar$ under $\iota_0, \iota_\infty$.

Note that the fibers of $\pi_\cochar : E_\cochar \to \mathbb{P}^1$ are symplectomorphic to $(X,\omega)$. The differential of $\pi_\cochar$ defines the vertical tangent bundle $\ker(d\pi_\cochar) =: T^{\mathrm{vert}} E_\cochar \hookrightarrow T E_\cochar$. A vertical almost complex structure $\widetilde{J}$ for the fibration $E_\cochar$ is an automorphism of $T^{\mathrm{vert}} E_\cochar$ such that $(\widetilde{J})^2 = - \mathrm{id}$, compatible with fiberwise symplectic forms. See \cite[Section 8.2]{MS12} for the theory of $\widetilde{J}$-holomorphic section counting in Hamiltonian fibrations.

\subsubsection{Equivariant maps and shift operators}
The un-normalized shift operators will be defined by counting $\widetilde{J}$-holomorphic sections of $E_\cochar$. Fix an almost complex structure $\widetilde{J}$ for the fibration $E_\cochar$, coming from a vertical family of almost complex structures. An inhomogeneous term for $E_\cochar$ is  a section (supported away from the neighborhood of $0, \infty \in C$):
\begin{equation}
    \nu_{E} \in C^\infty \left( C \times E_\cochar ; \mathrm{Hom}^{0,1} (TC, T^{\mathrm{vert}}E_\cochar) \right).
\end{equation}
Using the action $\rho_E$, one can define the equivariant perturbation data $\widetilde{J}^{eq}, \nu_E^{eq}$ as a equivariant family of almost complex structures inhomogeneous terms satisfying the equivariance conditions as in \cref{defn:eq-J} and \cref{defn:eq-nu}.
\begin{defn}\label{defn:eq-moduli-sections}
The \emph{equivariant moduli space of sections} is the space of maps
\begin{equation}
    \mathcal{M}^{eq}(E_\cochar) = \left \{ u : C \to E_\cochar, \  (w, w') \in E\sg \times ET : \pi_\beta(u(z)) = z, \ \overline{\partial}_{\widetilde{J}_{w'}} u = \nu_{E, w, w'}^{eq} \right\}.
\end{equation}
\end{defn}

For each homology class $\widetilde{A} \in H_2(E_\cochar ; \mathbb{Z})$ of a \emph{section} (that is, $(\pi_\cochar)_*[\widetilde{A}] = [\mathbb{P}^1]$) there is a corresponding component of the moduli space where the degree of the map is prescribed by $u_*[C] = \widetilde{A}$. For the fixed cocharacter $\cochar$, one can fix the \emph{minimal section} $s_{\mathrm{min}}$ given by a fixed point $x_{\mathrm{min}} \in X^T \subseteq X^\beta$ where the moment map for $\rho \circ \beta$ attains a minimum over $X^T$ (which we assume to be compact).

One can correspondingly define for $A \in H_2(X;\mathbb{Z})$ the moduli space
\begin{equation}\label{eqn:section-moduli}
    \mathcal{M}^{eq}_A (E_\cochar) = \left\{ (u, w, w') \in \mathcal{M}^{eq}(E_\cochar) : u_*[C] = s_{\mathrm{min}} + (\iota_0)_*A \right\}.
\end{equation}
The moduli space is equipped with an evaluation map
\begin{equation}
    \mathrm{ev}_{0, \infty} : \mathcal{M}^{eq}_A(E_\cochar) \to E(\sg \times T) \times X_0 \times X_\infty
\end{equation}
which is equivariant with respect to the action of $\sg \times T$ that acts on $E(\sg \times T) \times X_0 \times X_\infty$ by $(\eta, \theta) \cdot (w,w'; x_0, x_\infty) = (\eta w, \theta w'; \rho_0^{-1} x_0, \rho_\cochar^{-1} x_\infty)$.

Denote the $(\sg \times T)$-equivariant cohomology of $X$ with actions $\rho_0$ and $\rho_\cochar$ by $H^*_{\sg \times T}(X | \rho_0)$, $H^*_{\sg \times T}(X | \rho_\cochar)$ respectively. Note that by K\"unneth formula $H^*_{\sg \times T}(X | \rho_0) \cong H^*_T (X | \rho) \zpeq$ since $\sg$ acts trivially on $X$ under $\rho_0$. Fix a cohomology class $b_0 = \mathrm{PD}[Y_0] \in H^*_{\sg \times T}(X| \rho_0)$ represented by a $T$-invariant locally finite cycle, and similarly fix $b_\infty = \mathrm{PD}[Y_\infty] \in H^*_{\sg \times T}(X|\rho_\cochar)$ represented by a $T$-invariant compact cycle. From the inclusion $\mathcal{Y}(E_\cochar) = Y_0 \times Y_\infty \subseteq X_0 \times X_\infty$ we construct an equivariant incidence cycle $\mathcal{Y}^{eq}(E_\cochar)$, cf. \cref{defn:eq-incidence-cycle} and \cref{defn:totaleq-incidence-cycle}.

Correspondingly there is a moduli space
\begin{equation}
    \mathcal{M}^{eq}_A(E_\cochar) \pitchfork \mathcal{Y}^{eq}(E_\cochar) = \left\{ (u, w, w') \in \mathcal{M}^{eq}_A(E_\cochar) : \mathrm{ev}_{0, \infty}(u, w, w') \in \mathcal{Y}^{eq}(E_\cochar) \right\}
\end{equation}
and its restrictions $\mathcal{M}^{eq;i;2j}_A(E_\cochar) \pitchfork \mathcal{Y}^{eq;i;2j}(E_\cochar)$; cf. \eqref{eqn:i/2j-th-eq-moduli-space-qst}.

\begin{assm}\label{assm:shift-regular}
    The moduli space $\mathcal{M}^{eq;i;2j}_A(E_\cochar) \pitchfork \mathcal{Y}^{eq;i;2j}(E_\cochar)$ is regular of dimension $i + \sum_{k=1}^r 2j_k + \dim_{\mathbb{R}} X +2 c_1^{\mathrm{vert}}(s_\mathrm{min}+A) - |b_0| - |b_\infty|$, and can be compactified by stable maps and degeneration of equivariant parameters.  In particular, when the dimension is $0$, it gives a well-defined count in $\mathbb{F}_p$ (cf. \cref{lem:qst-Fp-count}).
\end{assm}

\begin{rem}
    \cref{assm:shift-regular} is a claim about compactness of the moduli spaces with incidence constraints, and will only hold for special choices of cocharacters $\cochar:S^1 \to T$ allowing maximum principle; see \cite[Section 1.12]{Rit14}, \cite[Remark 3.10]{Iri17} for a discussion. For the example we are interested in computing, we will verify this compactness assumption in \cref{sec:moduli-regular} for $\cochar$ satisfying certain conditions on weights, see \cref{defn:cochar-nonneg}. In practice, \cref{assm:shift-regular} means that the shift operators can be defined without localization only for special choices of $\cochar$. In general, defining the shift operators require passing to $H^*_T(X)_{\mathrm{loc}}$ for its definition.
\end{rem}

Denote by $H^*_{\sg \times T}(X;\Lambda |\rho)$ the graded completed tensor product $H^*_{\sg \times T}(X;\Lambda |\rho) := H^*_{\sg \times T} (X|\rho) \ \widehat{\otimes}_{\mathbb{F}_p} \Lambda$.

\begin{defn}\label{defn:bare-shift}
Fix a cocharacter $\cochar: S^1 \to T$ satisfying \cref{assm:shift-regular}. The \emph{bare shift operators} associated to $\cochar$ is a map
\begin{equation}
    S_\cochar : H^*_{\sg \times \tg} (X ; \Lambda| \rho_0) \to H^*_{\sg \times \tg} (X ; \Lambda | \rho_\cochar)
\end{equation}
defined for $b_0 \in H^*(X;\mathbb{F}_p)$ as the sum
\begin{equation}
    S_\cochar(b_0) = \sum_{b_\infty} \sum_A \sum_{i \ge 0, j \in (\mathbb{Z}_{\ge 0})^r} \left( \#  \mathcal{M}^{eq; i;2j}_A(E_\cochar) \pitchfork \mathcal{Y}^{eq; i;2j}(E_\cochar) \right) q^A (t,\theta)^i\  {\teq}^{j} \  b_\infty^\vee \in H^*_{\sg \times T}(X;\Lambda|\rho_\cochar)
\end{equation}
and extended $q, t, \teq$-linearly. The sum $A$ is over $A \in H_2(X;\mathbb{Z})$, where the sections are counted in degree $s_{\mathrm{min}} + (\iota_0)_* A$, see \eqref{eqn:section-moduli}.
\end{defn}

The bare shift operators are $H^*_{\sg \times T}(\mathrm{pt})$-linear by definition. The bare shift operators were referred to as \emph{equivariant quantum Seidel maps} in \cite{LJ20}, \cite{LJ21}.

Consider the group homomorphism $\phi_\cochar : \sg \times \tg \to \sg \times \tg$ defined by $\phi_\cochar(\eta, \theta) = (\eta, \theta + \cochar(\eta))$. The identity map $\mathrm{id}_X: (X, \rho_0) \to (X, \rho_\cochar)$ is equivariant with respect to $\phi_\cochar$ so that there exists a map from functoriality of equivariant cohomology:
\begin{defn}\label{defn:eq-reparam-map}
The \emph{equivariant re-parametrization map} associated to $\cochar$ is the map
\begin{equation}
    \Phi_\cochar := (\mathrm{id}_X, \phi_\cochar)^* : H^*_{\sg \times \tg} (X ; \Lambda | \rho_\cochar) \to H^*_{\sg \times \tg} (X ; \Lambda | \rho_0).
\end{equation}
\end{defn}

\begin{defn}\label{defn:shift}
The \emph{shift operator} associated to $\cochar$ is the composition
\begin{equation}
    \mathbb{S}_\cochar : = \Phi_\cochar \circ S_\cochar : H^*_{\sg \times \tg} (X;\Lambda |\rho_0) \to H^*_{\sg \times \tg} (X;\Lambda |\rho_0).
\end{equation}
\end{defn}

The shift operator $\mathbb{S}_\cochar$ is a ($\cochar$-)twisted-linear homomorphism in the sense that
\begin{equation}\label{eqn:shifted-linear}
    \mathbb{S}_\cochar \left( f(\loopeq,   \teq) \alpha \right) = f(\loopeq,   \teq + \cochar \loopeq) \ \mathbb{S}_\cochar\left(\alpha\right)
\end{equation}
for any polynomial $f$ in the variables $\loopeq, \teq$. In particular, it is \emph{not} $H^*_{\sg \times \tg}(\mathrm{pt})$-linear.

The following is the main theorem of \cite{LJ20}, reproved in \cite{LJ21}. In the algebraic geometry context, the proof is originally due to \cite{OP10} for Hilbert schemes of points in $\mathbb{C}^2$, see \cite[Section 8]{MO19} for a systematic discussion. 

\begin{thm}[{\cite[Theroem 3.6]{LJ21}}]\label{thm:shift-compatibility-qconn}
    The shift operators $\mathbb{S}_\cochar$ commute with the (equivariant) quantum connection $\nabla_a$. Equivalently, the bare shift operators $S_\beta$ commute with $\nabla_a$ up to a shift in the equivariant parameters:
    \begin{equation}
        \nabla_a |_{\lambda \mapsto \lambda - \cochar t} \circ S_\cochar = S_\cochar \circ \nabla_a.
    \end{equation}
\end{thm}

\subsection{Compatibility of quantum Steenrod and shift operators}\label{ssec:operators-compatibility}
In this section we prove that the quantum Steenrod operators $\Sigma_b$ and the shift operators $\mathbb{S}_\cochar$ commute. The proof follows a familiar strategy in symplectic topology, by considering $1$-dimensional moduli spaces corresponding to a degeneration of solutions, which allows one to identify the counts from the $0$-dimensional moduli spaces arising as the boundary.



\subsubsection{Quantum Steenrod operations for the $\infty$-fiber}\label{sssec:qst-infty-fiber}

We define the twisted versions of the $T$-equivariant quantum Steenrod operations 
\begin{equation}
    \Sigma_b^\cochar : H^*_{\sg \times T}(X;\Lambda|\rho_\cochar) \to H^*_{\sg \times T}(X;\Lambda|\rho_\cochar)
\end{equation}
obtained as endomorphisms of $H^*_{\sg \times T} (X)$ where $X$ is equipped with the twisted action $\rho_\cochar(\eta, \theta) (x) = \rho_0(\eta, \theta - \cochar(\eta))(x)$. The structure constants for $\Sigma^{\cochar}_b$ are obtained by counting the moduli spaces
\begin{equation}
    \mathcal{M}^{eq}_\cochar := \{ \left(u : C \to X; (w, w') \in E(\sg \times \tg) \right) : \overline{\partial}_{J_w'}u = \nu^{eq, \cochar}_{w, w'}\}
\end{equation}
with incidence constraints under evaluation maps $\mathcal{M}^{eq}_\cochar \to E(\sg \times \tg) \times (X \times X^p \times X)$.
Consider the action (cf. \eqref{eqn:X^p-cyclic-permutation} where $\sigma_X$ is defined as the cyclic permutation action on $X^p$)
\begin{align}
    \sigma_X^\cochar : (\sg \times T) \times X  \times X^p \times X &\to X \times X^p \times X \\
    [(\eta, \theta); (x_0; x_1, \dots, x_p; x_\infty)] &\mapsto \rho_\cochar(\eta, \theta) \cdot (x_0;  x_{1-\eta}, \dots, x_{p-\eta}; x_\infty).
\end{align}
where $\rho_\cochar$ acts diagonally. The action $\sigma_X^\cochar$ is defined so that $(\mathrm{id}, \phi_\cochar): (X \times X^p \times X; \rho \circ \sigma_X) \to (X \times X^p \times X; \sigma_X^\cochar)$ is a $(\sg \times T)$-equivariant map. Using the action $\sigma_X^\cochar$, the evaluation map becomes equivariant and one can define the moduli spaces with incidence constraints to define the structure constants of $\Sigma_b^\cochar$. The following observation shows that the twisted quantum Steenrod operations are nothing but conjugation of the ordinary quantum Steenrod operations by the twisted-linear isomorphism $\Phi_\cochar$.

\begin{lem}\label{lem:conjugation-reparam}
    The following equality holds for two maps $H^*_{\sg \times \tg} (X_0; \Lambda| \rho_0) \to H^*_{\sg \times \tg} (X_0; \Lambda| \rho_0)$:
    \begin{equation}
        \Phi_\cochar \circ \Sigma^{\cochar}_{b} \circ \Phi_\cochar^{-1} = \Sigma_b.
    \end{equation}
    That is, the following diagram commutes.
    \begin{center}
    \begin{tikzcd}
        H^*_{\sg \times \tg}(X;\Lambda |\rho_\cochar) \dar["\Phi_{\cochar}"] \rar["\Sigma_{b}^{\cochar}"] & 
        H^*_{\sg \times \tg}(X;\Lambda |\rho_\cochar) \dar["\Phi_{\cochar}"] \\
        H^*_{\sg \times \tg}(X;\Lambda |\rho_0) \rar["\Sigma_b"] &
        H^*_{\sg \times \tg}(X; \Lambda | \rho_0)
    \end{tikzcd}.
\end{center}
\end{lem}
\begin{proof}
    For a cohomology class $b_0, b, b_\infty \in H^*_{\sg \times \tg} (X|\rho_0)$ fix an equivariant incidence cycle $\mathcal{Y}^{eq} = \{\mathcal{Y}_{w, w'}\}$ that represents $b_0 \otimes b^{\otimes p} \otimes b_\infty$ via Poincar\'e duality, see \cref{defn:totaleq-incidence-cycle}. The equivariant incidence cycle satisfies the ordinary equivariance condition $\mathcal{Y}_{ \eta \cdot w, \theta \cdot w'} = \rho_0(\eta, \theta)^{-1} \sigma_X(\eta)^{-1} \mathcal{Y}_{ w, w'}$ from \cref{defn:totaleq-incidence-cycle}.

 Consider a copy of $E(\mathbb{Z}/p \times T) = S^\infty \times (S^\infty)^r$ where the action of $\mathbb{Z}/p \times T$ is multiplication by $(\eta, \theta - \cochar(\eta))$. This twisted multiplication still defines a free action, so $E(\mathbb{Z}/p \times T)$ quotiented by this action gives a model for the classifying space of $\mathbb{Z}/p \times T$. To distinguish the actions, denote this model by $E^\cochar(\mathbb{Z}/p \times T)$ and the original model (where $\sg \times T$ acts by usual multiplication) by $E^0(\mathbb{Z}/p \times T)$. 

     Now let $\mathcal{Z}_{w, w'} := \mathcal{Y}_{w,w'}$ so that tautologically 
     \begin{equation}
         \mathcal{Z}_{\eta \cdot w, (\theta - \cochar \eta)\cdot w'} = \rho_0(\eta, \theta - \cochar \eta)^{-1}\sigma_X(\eta)^{-1} \mathcal{Z}_{w, w'} = \sigma_X^{\cochar}(\eta, \theta)^{-1} \mathcal{Z}_{w, w'}.
     \end{equation} 
     Hence $\mathcal{Z}^{eq}$ defines a $(\mathbb{Z}/p \times T)$-equivariant incidence cycle for the model $E^\cochar(\mathbb{Z}/p \times T)$, representing the cohomology class $\Phi_\cochar^{-1} (b_0 \otimes b^{\otimes p} \otimes b_\infty) \in H^*_{\sg \times T} (X \times X^p \times X|\sigma_X^\cochar)$. 
     
     In the same way, one can choose an equivariant inhomogeneous term (\cref{defn:eq-nu}) for $(X, \rho_0)$ with the model $E^0(\sg \times T)$, which can tautologically be interpreted as an equivariant inhomogeneous term for $(X, \rho_\cochar)$ with the model $E^\cochar(\sg \times T)$.
     In particular, one can use this inhomogeneous term and the equivariant incidence cycle to define the twisted quantum Steenrod operations, so that
    \begin{equation}
        \Sigma_b^\cochar ( \Phi_\cochar^{-1} b_0) = \sum \Phi_\cochar^{-1}(b_\infty)^\vee \  \left(\# \ \mathcal{M}_{\cochar, A}^{eq; i; 2j} \pitchfork \mathcal{Z}^{eq;i;2j}\right) \ q^A \ (t,\theta)^i \ 
        \lambda_\cochar^j \in H^*_{\mathbb{Z}/p \times T}(X;\Lambda|\rho_\cochar)
    \end{equation}
    where $\lambda_\cochar \in H^*_{\mathbb{Z}/p \times T}(\mathrm{pt})$ is such that $\Phi_\cochar(\lambda_\cochar) = \lambda$, i.e. $\lambda_\cochar = \lambda - \cochar t$. Compare this with
    \begin{equation}
        \Phi_\cochar^{-1}(\Sigma_b(b_0)) = \sum \Phi_\cochar^{-1}(b_\infty^\vee) \  \left(\# \ \mathcal{M}_{A}^{eq; i; 2j} \pitchfork \mathcal{Y}^{eq;i;2j}\right) \ q^A \ (t,\theta)^i \ 
        (\Phi_\cochar^{-1} \lambda)^j \in H^*_{\mathbb{Z}/p \times T}(X;\Lambda|\rho_\cochar).
    \end{equation}
    Since $\mathcal{Z}_{w, w'} = \mathcal{Y}_{w, w'}$, we tautologically have
    \begin{equation}
        \left(\# \ \mathcal{M}_{\cochar, A}^{eq; i; 2j} \pitchfork \mathcal{Z}^{eq;i;2j}\right) = \left(\# \ \mathcal{M}_A^{eq; i; 2j} \pitchfork \mathcal{Y}^{eq;i;2j}\right),
    \end{equation}
    and the coefficients agree. Since $\Phi_\cochar$ preserves the Poincar\'e pairing, we have $ \Phi_\cochar^{-1}(b_\infty^\vee) = \Phi_\cochar^{-1}(b_\infty)^\vee$. Finally, the shifting property $\Phi_\cochar^{-1}(\lambda) = \lambda - \cochar t$ implies that the two operations agree, as desired.
    \end{proof}

\subsubsection{A TQFT argument}\label{sssec:qst-TQFT}
It is convenient to introduce the following family of curves. Fix $S = \mathbb{P}^1$, and take the trivial family of curves $C \times S \to S$ where $C$ is our distinguished parameterized copy of $\mathbb{P}^1$ which is the domain of the maps we are counting. Let $\mathcal{C}$ be the blowup of $C \times S$ at two points $(0,0)$, $(\infty, \infty)$. The resulting family $\mathcal{C} \to S$ of nodal genus $0$ curves over $S$ has two nodal fibers at $0, \infty \in S$ and smooth fibers at all other values of $v \in S$. Let $z_* : S \to \mathcal{C}$ be a section of this family defined as the proper transform of the diagonal section. Also fix $z_0, z_\infty: S \to \mathcal{C}$ the sections defined as the proper transform of the constant sections $\{0\} \times S$, $\{\infty\} \times S$. Let $L \subseteq S$ be the positive real line compactified by $0, \infty \in S$, oriented from $0$ to $\infty$.

Fix the identification $\mathcal{C}_v \cong C$ for $v \neq 0, \infty$. For $v = 0$ and $v= \infty$, the fiber $\mathcal{C}_v$ is a nodal genus $0$ curves with two components $\mathcal{C}_{v, +} \cong C$ and $\mathcal{C}_{v, -}$. We moreover fix the biholomorphism of the exceptional curve $\mathcal{C}_{0, -} \cong C$, so that $0, 1, \infty \in C$ corresponds to $z_0(0)$, $z_*(0)$, and the nodal point respectively. Similarly, fix the biholomorphism $\mathcal{C}_{\infty, -} \cong C$, so that $0, 1, \infty \in C$ corresponds to the nodal point, $z_*(\infty)$, $z_\infty(\infty)$ respectively. For $z \in \mathcal{C}$, denote by $\hat{z}$ the image under projection $\mathcal{C} \to C \times S \to C$; if $z \in \mathcal{C}_{v, -}$ for $v =0 $ or $v = \infty$, then $\hat{z} = v$, but otherwise $\hat{z} = {z}$ under the identification $C \cong \mathcal{C}_v$. Denote the complement in $\mathcal{C}$ of the two nodal points by $\mathcal{C}^{\mathrm{reg}}$.

There is a natural action $\tau_\mathcal{C}$ of $\sg$ on $\mathcal{C}$ induced from $\tau_C$ acting on $C$. It acts by fiberwise rotation for the family $\mathcal{C} \to S$. The nodal fibers at $0$ and $\infty$ have two components, which are simultaneously rotated by the action of $\tau_\mathcal{C}$.

Take the equivariant inhomogeneous term
\begin{equation}
    \nu_{\mathcal{C}}^{eq} \in C^\infty \left(E (\sg \times \tg) \times \mathcal{C} \times E_\cochar ; \mathrm{Hom}^{0,1}(T \mathcal{C}^{\mathrm{reg}}/S, T^{\mathrm{vert}}E_\cochar) \right),
\end{equation}
supported away from the nodes, also satisfying the equivariance condition
\begin{equation}
    D \rho_{E}^{-1} \circ (\nu_{\mathcal{C}}^{eq})_{w, w', \tau_{\mathcal{C}}(z), \rho_{E}(x) } \circ D \tau_{\mathcal{C}} = (\nu_{\mathcal{C}}^{eq})_{(\eta, \theta) \cdot (w, w'), z, x}.
\end{equation}

\begin{lem}\label{lem:tqft-qst-and-shift}
    The following equality holds for two maps $H^*_{\sg \times \tg} (X_0; \Lambda | \rho_0) \to H^*_{\sg \times \tg} (X_\infty; \Lambda | \rho_\cochar)$:
    \begin{equation}
        \Sigma_{b}^{\cochar} \circ S_\cochar = S_\cochar \circ \Sigma_{b}.
    \end{equation}
\end{lem}
\begin{proof}
Consider the moduli space
\begin{equation}
    \mathcal{N}^{eq}_{L}(E_\cochar) = \left\{ v \in L, \ u : \mathcal{C}_v \to E_\cochar, \ (w, w') \in E(\sg \times \tg) : \pi_\beta (u(z)) = \hat{z}, \  \overline{\partial}_{\widetilde{J}} u = \nu_{\mathcal{C}_v, w, w'}^{eq} \right\}
\end{equation}
together with evaluation maps
\begin{align}
    \mathrm{ev}_{L}(E): \mathcal{N}^{eq}_L(E_\cochar) &\to E(\sg \times \tg) \times X_0 \times E_\cochar^p \times X_\infty \\
    (v, u, w, w') &\mapsto \left( (w,w') ;\  u(z_0(v)) ;\  u(\tau_{\mathcal{C}} z_\ast(v)), \dots, u(\tau_{\mathcal{C}}^p z_\ast(v)) ;\  u(z_\infty(v)) \right).
\end{align}
By Mayer--Vietoris (see \cite[Lemma 3.7]{Iri17}, \cite[Lemma A.2]{LJ21}), there exists a cohomology class $\hat{b} \in H^*_{\sg \times T}(E_\cochar| \rho_E)$ such that $\iota_0^*\hat{b} = b$, $\iota_\infty^* \hat{b} = \Phi_\cochar^{-1} b$. Fix an equivariant incidence cycle (cf. \cref{defn:totaleq-incidence-cycle}) $\hat{\mathcal{Y}}^{eq}$ in $X_0 \times E_\cochar^p \times X_\infty$ corresponding to $p$ copies of $\hat{b} \in H^*_{\sg \times T}(E_\cochar| \rho_E)$ together with $b_0 \in H^*_{\sg \times T}(X_0 |\rho_0)$, $b_\infty \in H^*_{\sg \times T}(X | \rho_\cochar)$.

The incidence cycles define a moduli space $\mathcal{N}_L^{eq}(E_\cochar) \pitchfork \hat{\mathcal{Y}}^{eq}$ with the incidence constraints at the distinguished marked points $z_0(v), \tau_\mathcal{C}^j z_*(v), z_\infty(v) \in \mathcal{C}_v$ of the domain curve. We may choose generically the equivariant homogeneous term $\nu_{\mathcal{C}}^{eq}$ to make all equivariant moduli spaces regular. 

Now consider the moduli space $\mathcal{N}_L^{eq}(E_\cochar) \pitchfork \hat{\mathcal{Y}}^{eq}$ and restrict to equivariant cells $(w, w') \in \Delta^i \times \Delta^{2j_1} \times \cdots \times \Delta^{2j_r}$ so that the corresponding moduli space is $1$-dimensional. Then the only way this moduli space fails to yield a cobordism in the usual sense between the $0$-dimensional moduli spaces with $v \in L$ specialized to be $v=0$ and $v=\infty$, for $\partial L = \infty - 0$, is by the parameter $w \in \Delta^i$ approaching $\partial\Delta^i$. By $\mathbb{Z}/p$-equivariance, the counts arising from such strata yield counts divisible by $p$. Therefore, the counts of the $0$-dimensional moduli spaces with $v = 0, \infty$ are equal mod $p$.

The solutions with $v = 0$ correspond to pair of solutions $u_{0, -}: \mathcal{C}_{0, -} \to X_0$, $u_{0,+}: \mathcal{C}_{0, +} \to E_\cochar$ where (i) $u_{0, -}: \mathcal{C}_{0, -} \to X_0$ satisfies incidence constraints at the $p$ marked points $\tau_{\mathcal{C}}^j z_*(v)$ given by Poincar\'e dual of $\iota_0^* \hat{b} = b$, and (ii) $u_{0, +}: C \to E_\cochar$ is a section of $E_\cochar \to \mathbb{P}^1$ with incidence constraints at two marked points. Each of these counts corresponds to (i) quantum Steenrod operations $\Sigma_b$ for $X_0$, and (ii) shift operators, respectively. Hence, the moduli space with $v = 0$ gives the structure constants $(b_\infty, S_\cochar \circ \Sigma_b(b_0))$ of the composition $S_\cochar \circ \Sigma_b$. Similarly, by \cref{lem:conjugation-reparam}, the solutions with $v = \infty$ count the structure constants $(b_\infty, \Sigma_b^\cochar \circ  S_\cochar(b_0))$. The cobordism then implies that these structure constants agree, which is the desired result.
\end{proof}

\begin{thm}\label{thm:compatibility}
    The quantum Steenrod operations $\Sigma_b$ and the shift operators $\mathbb{S}_\cochar$ commute. Equivalently, the bare shift operators $S_\beta$ commute with $\Sigma_b$ up to a shift in the equivariant parameters:
        \begin{equation}
        \Sigma_b |_{\lambda \mapsto \lambda - \cochar t} \circ S_\cochar = S_\cochar \circ \Sigma_b.
    \end{equation}
\end{thm}
\begin{proof}
    Recall that the shift operators are defined as the composition $\mathbb{S}_\cochar = \Phi_\cochar \circ S_\cochar$, so the claim of the commutativity amounts to
    \begin{equation}
        \Sigma_b \circ \Phi_\cochar \circ S_\cochar = \Phi_\cochar \circ S_\cochar \circ \Sigma_b .
    \end{equation}
    This is a consequence of \cref{lem:conjugation-reparam} and \cref{lem:tqft-qst-and-shift}.
\end{proof}

\section{Regularity of moduli spaces}\label{sec:moduli-regular}
In this section, we establish the necessary transversality and compactness properties for the moduli spaces of (perturbed) $J$-holomorphic curves considered for the definition of quantum Steenrod operations and shift operators, focusing on the example of the cotangent bundle of the complete flag variety $X=T^*(G/B)$. We show that the compactness properties can be verified for moduli spaces when the incidence constraints are given by certain locally finite \emph{relative} cycles intersecting on a compact subset. These relative cycles will be obtained in the case of $X = T^*(G/B)$ by the conormals to Schubert varieties.

\subsection{Springer resolution}\label{ssec:moduli-springer-resolution}
In this section, we gather the facts about the geometry and topology of the Springer resolution $T^*(G/B)$ necessary for showing the regularity of the moduli spaces in the definition of ($\T$-equivariant) quantum Steenrod operations and the shift operators. In particular, we review the (i) projection to the nilpotent cone, which allows the use of maximum principle, and (ii) the conormals to Schubert varieties, which provide a natural geometric basis in cohomology.

\subsubsection{Geometry and topology}
The facts introduced here are classical, see e.g. \cite{Bri05} and \cite{CG97} for references.

Let $G$ be a complex semisimple connected Lie group and $B$ be a Borel subgroup, and $T_\mathbb{C}$ be the maximal torus in $B$. Correspondingly consider the Lie algebras $\mathfrak{g}$, $\mathfrak{b}$, $\mathfrak{h}$. Using the Killing form, we will identify $\mathfrak{g}^* \cong \mathfrak{g}$ interchangeably. Note that $G$ acts transitively on the set of Borel subalgebras of $\mathfrak{g}$ by conjugation, with $B$ being the stabilizer of $\mathfrak{b}$. Therefore the set of Borel subalgebras is a homogeneous space $G/B$, the \emph{flag variety}. The flag variety is a smooth Fano variety, hence a monotone symplectic manifold by the anticanonical polarization. The group $G$, and the maximal torus subgroup $T_{\mathbb{C}}$, naturally acts on $G/B$.

The holomorphic cotangent bundle $X = T^*(G/B)$ of the flag variety is a holomorphic symplectic manifold equipped with the standard holomorphic symplectic form $\omega_{\mathbb{C}}$. By restriction from the projectivization $\mathbb{P}(T^*G/B \oplus \underline{\mathbb{C}}) \to G/B$, the cotangent bundle $T^*(G/B)$ also admits a K\"ahler structure $\omega$ compatible with its integrable complex structure. The K\"ahler form $\omega$ is indeed non-exact, with the zero section $G/B$ being a symplectic submanifold, and $T^*(G/B)$ with its integrable complex structure does admit many holomorphic curves. This is in contrast with the symplectic structure such as $\mathrm{Re}(\omega_{\mathbb{C}})$, for which the zero section $G/B$ is a Lagrangian submanifold.

\begin{rem}
    There is also a complete hyperK\"ahler metric on $X$ such that the metric restricts to the homogeneous metric on the zero section $G/B$ and is compatible with the standard holomorphic symplectic form $\omega_\mathbb{C}$ \cite{Nak94}, \cite{Biq96}.
\end{rem}

For the basepoint $e = B/B \in G/B$, note that the cotangent fiber $T^*_e(G/B) \cong (\mathfrak{g}/\mathfrak{b})^* \cong \mathfrak{b}^\perp \cong \mathfrak{n}$, where $\mathfrak{n}$ is the nilpotent subalgebra in the decomposition of the Borel $\mathfrak{b} \cong \mathfrak{h} \oplus \mathfrak{n}$. One can then identify the cotangent bundle with the vector bundle obtained from the principal $B$-bundle $G \to G/B$ with fibers isomorphic to $\mathfrak{n}$, that is $T^*(G/B) \cong G \times^B \mathfrak{n}$ \cite[Lemma 3.2.2]{CG97}. Denote by $\mathcal{N} \subseteq \mathfrak{g} \cong \mathfrak{g}^*$ the cone of nilpotent elements. The description
\begin{equation}
    T^*(G/B) \cong G \times^B \mathfrak{n} := \{ (gB, x) : g^{-1}xg \in \mathfrak{n} \} \subseteq G/B \times \mathcal{N}
\end{equation}
implies that there is a projection $T^*(G/B) \to \mathcal{N}$ which is proper and birational. Birationality follows from the fact that there is a unique Borel subalgebra containing a regular nilpotent element, which forms a Zariski dense subset of $\mathcal{N}$.

\begin{defn}\label{defn:springer-resolution}
    The \emph{Springer resolution} is the holomorphic cotangent bundle $T^*(G/B)$ considered as a resolution of singularities by the map $T^*(G/B) \to \mathcal{N}$.
\end{defn}
The $G$-action on $G/B$ induces a complex Hamiltonian $G$-action on $X = T^*(G/B)$ with a complex moment map $\mu : X \to \mathfrak{g}^*$. The Springer resolution $T^*(G/B) \to \mathcal{N}$ can also be identified with this complex moment map \cite[Proposition 1.4.10]{CG97}.

There is also a $\mathbb{C}^*$-action which acts by dilating the cotangent fibers. The $S^1$ subgroup rotating the fibers acts in a Hamiltonian way for the K\"ahler structure $\omega$. Together with the action of the compact part $T$ of the maximal torus $T_\mathbb{C}$, consider the $\mathbf{T} = T \times S^1$ action on $X = T^*(G/B)$. Denote the ground ring of equivariant cohomology as
\begin{equation}
    H^*_\T(\mathrm{pt}) = \mathbb{F}_p [\![ \teq, \seq ]\!] := \mathbb{F}_p [\![\teq_1, \dots, \teq_r, h]\!]
\end{equation}
where $\teq$ is a shorthand for the collection of all equivariant parameters $\teq_1, \dots, \teq_r$ for the $T$-action, and $\seq$ denotes the equivariant parameter for the rotating $S^1$-action.

\begin{rem}
    It is essential to work equivariantly with respect to the $S^1$-action rotating the fibers for Gromov--Witten theory to be nontrivial for $X$. The ordinary Gromov--Witten theory for $X$ is trivial because $X$ admits a deformation to an affine variety, by Grothendieck's simultaneous resolution. Such deformation cannot be chosen in a way preserving the $S^1$-symmetry, and accordingly the (reduced) Gromov--Witten theory may be nontrivial, see \cite[Section 4.2]{BMO11}.
\end{rem}

\begin{exmp}
    For $G = \mathrm{SL}_2(\mathbb{C})$, we may take the subgroup $B$ of upper triangular matrices, with $T_\mathbb{C}$ the diagonal matrices. The homogeneous space $G/B$ is isomorphic to $\mathbb{P}^1$. On the level of Lie algebras, $\mathfrak{g} = \mathfrak{sl}_2 = \{ \begin{pmatrix} a & b \\ c & -a \end{pmatrix} \}$ is the Lie algebra of traceless matrices, and the nilpotent cone $\mathcal{N}$ can be identified with matrices with vanishing determinant $a^2 + bc = 0$.

    In particular, the nilpotent cone $\mathcal{N} \cong \{ a^2 + bc = 0 \}$ is a copy of a singular quadric cone in ambient dimension 3. Its resolution of singularity is given by the blowup of the singular point, which is identified with $T^*\mathbb{P}^1$; the zero section $\mathbb{P}^1$ corresponds to the exceptional divisor of the blowup.

    There is a $\T = S^1 \times S^1$ action on the resolution $T^*\mathbb{P}^1$; the first factor $T=S^1$ acts on the base by rotating the $\mathbb{P}^1$ along a fixed axis extended to an action on the cotangent bundle, and the latter factor $S^1$ acts on $T^*\mathbb{P}^1$ by rotating the cotangent fibers.
\end{exmp}


The (equivariant) quantum Steenrod operations and shift operators are defined on the underlying module of equivariant quantum cohomology. Hence we consider the equivariant cohomology $H^*_\mathbf{T}(X)$ of $X = T^*(G/B)$. The deformation retract onto the zero section $G/B$ is $\mathbf{T}$-equivariant, so it suffices to understand $H^*_\mathbf{T}(G/B)$.

The flag variety $G/B$ admits a complex cell decomposition \cite[Corollary 3.1.12]{CG97} into \emph{Schubert cells} $\mathring{X}_w = BwB/B$ indexed by elements of the Weyl group $w \in W = N_G(T)/T$. The cells $\mathring{X}_w$ are affine spaces of complex dimension $\ell(w)$ for the length function $\ell: W \to \mathbb{Z}_{\ge 0}$. The closures $X_w = \overline{\mathring{X}_w}$ are the \emph{Schubert varieties} which are themselves union of Schubert cells, $X_w = \bigsqcup_{v \le w} \mathring{X}_v$ for $v \le w$ in a partial order on $W$ known as the \emph{Bruhat order}. The fundamental classes $[X_w]$ yield an additive basis of cohomology $H^*(G/B)$. In particular, the cohomology $H^*(G/B;\mathbb{Z})$ is concentrated in even degrees, is torsion-free, and the classes are represented by genuine cycles (cf. \cref{assm:eq-formality}).

The Schubert cells can also be understood as the Morse cell decomposition from a Morse function obtained by the moment map for a generic cocharacter $\cochar: S^1 \to T$ in the dominant Weyl chamber. The critical points are the $T$-fixed points $wB/B$, whose stable manifold is exactly the Schubert cell $\mathring{X}_w$. The Schubert varieties are manifestly $T$-invariant and their fundamental classes generate $H^*_\T(G/B) \cong H^*_\T(X)$ as a $H^*_\T(\mathrm{pt})$-module.

There are also \emph{opposite Schubert cells} obtained by reflection by the longest Weyl element, $\mathring{X}^w := w_0 \mathring{X}_{w_0 w} = B^- w B/B$ for the opposite Borel $B^-$. Their closures are the opposite Schubert varieties. Morse theoretically, the opposite Schubert cells arise from negating the moment map. The following special case of a general intersection behavior between Schubert cells and the opposite Schubert cells is useful to note.

\begin{lem}[{\cite{Deo85}}]\label{lem:Schubert-transverse}
    The Schubert variety $X_w$ and the opposite Schubert variety $X^w$ intersect transversally at a single point. Indeed, the opposite Schubert varieties are intersection duals of the Schubert varieties for the Poincar\'e pairing, $[X_w]^\vee = [X^w]$. If $w < v$, the intersection of $X_w$ and $X^v$ is empty. 
\end{lem}

\begin{exmp}
    In the case of $T^*\mathbb{P}^1$, the Weyl group $W$ is identified with $\mathbb{Z}/2$ with $1 \in W$ and the nontrivial element $s \in W$. The Bruhat order is given by $1 < s$. The Weyl group elements are represented in $G = \mathrm{SL}_2(\mathbb{C})$ by permutation matrices.

    Correspondingly there are two fixed points of the $T=S^1$-action on $G/B$, given by the cosets $B/B$ and $sB/B$. Under the isomorphism $G/B \cong \mathbb{P}^1$, these correspond to the two fixed points of a rotation $S^1$-action (the north and south poles). Fixing a dominant cocharacter $\beta: S^1 \to T = S^1$ (in this case, this just means that the degree of $\beta$ is positive, the corresponding Schubert cells are given by $\mathring{X}_1$ = the north pole (a $0$-cell), and $\mathring{X}_s$ = the complement of the north pole (a $2$-cell). The corresponding Schubert varieties are $X_1 = \mathring{X}_1$ = the north pole, and $X_s = \mathring{X}_s \cup \mathring{X}_1$ = the whole $\mathbb{P}^1$.

    Opposite Schubert stratification is given by reversing the cocharacter $\beta$, so we have $\mathring{X}^1$ = the complement of the south pole, $\mathring{X}^s$ = the south pole. Indeed, since the two poles are disjoint, $X_1 \cap X^s = \emptyset$. 
\end{exmp}

\subsubsection{Cohomological stable envelopes}
Consider the conormals $N^*\mathring{X}_w \subset T^*(G/B)$, and their union $Y_w = \bigsqcup_{v \le w} N^*\mathring{X}_v \subseteq T^*(G/B)$. By a slight abuse of notation, we will refer to $Y_w$ as the conormals of the Schubert varieties $X_w$; clearly $Y_w$ retracts to $X_w$ for the retraction of $T^*(G/B)$ to the zero section.

The (locally finite) cycles $Y_w$ are holomorphic Lagrangian (locally finite) cycles of $X = T^*(G/B)$, which are manifestly $\mathbf{T}$-invariant hence representing equivariant cohomology classes in $H^*_\T(X)$. Denote by $\mathbb{Y} := \bigcup_{w \in W} Y_w$ the union of all conormals. 



Note that one can also start from the anti-dominant cocharacter and correspondingly define the opposite Schubert cells $\mathring{X}^w$, the opposite conormals $Y^w$, and the union of opposite conormals $\mathbb{Y}^- = \bigcup_{w \in W} Y^w$. \cref{lem:Schubert-transverse} implies that the intersection pairing between $H^*(X, X \setminus \mathbb{Y})$ and $H^*(X, X \setminus \mathbb{Y}^-)$ is upper-triangular in the basis of (Poincar\'e duals of) conormals $Y_w$ and $Y^w$, with diagonal entries given by $\pm 1$.

\begin{prop}\label{prop:conormals-span-cohomology}
    The conormal cycles $[Y_w]$ span the Poincar\'e dual of middle-dimensional relative cohomology $H^{\dim_{\mathbb{C}}(X)}(X,X \setminus \mathbb{Y})$ of rank $|W|$.
\end{prop}
\begin{proof}
    The conormal cycles $[Y_w]$ represent classes in the Borel--Moore homology of $\mathbb{Y} := \bigcup_{w \in W} Y_w$, i.e. classes in $H_*^{\mathrm{BM}}(\mathbb{Y}) \cong H^*(X, X\setminus\mathbb{Y})$ by Poincar\'e duality.

    Choose a filtration of $\mathbb{Y}$ into locally closed subsets $\emptyset = Y_0 \subseteq Y_1 \subseteq \cdots \subseteq Y_{|W|} = \mathbb{Y}$ such that every $Y_j \setminus Y_{j-1}$ is a conormal of a Schubert cell $N^*\mathring{X}_w$. Note that the Borel--Moore homology $H_*^{\mathrm{BM}}(Y_j \setminus Y_{j-1})$ is of rank $1$ and concentrated in degree $\dim_{\mathbb{C}}(X) = \frac{1}{2} \dim_{\mathbb{R}} (X)$, spanned by the class of  $[N^* \mathring{X}_w]$. Iterative application of the long exact sequence in Borel--Moore homology (see e.g. \cite[(2.6.10)]{CG97}) implies that $H_*^{\mathrm{BM}}(\mathbb{Y}) \cong H^*(X, X \setminus \mathbb{Y})$ is concentrated in degree $\dim_{\mathbb{C}}(X)$ and is of rank $|W|$. Since for each $w \in W$, $Y_w = N^* \mathring{X}_w \cup \bigsqcup_{v < w} N^* \mathring{X}_v$, it follows that the Poincar\'e duals of $[Y_w]$ are linearly independent in $H^*(X, X \setminus \mathbb{Y})$, as desired.
\end{proof}

The equivariant cohomology classes in $H^*_\T(X)$ represented by the cycles $Y_w$ now factor through the map
\begin{equation}
     H^*_\T(X, X \setminus \mathbb{Y}) \to H^*_\T(X),
\end{equation}
in the long exact sequence of pairs.

In equivariant cohomology $H^*_\T(X)$, there is a distinguished basis constructed by \cite{MO19} called the \emph{stable envelopes}, which satisfy remarkable properties. The underlying idea has a Morse-theoretic flavor, which already features in \cite[Lemma 1.12]{MT06}.

To define the stable bases, note that the cotangent fiber at any $T$-fixed point $wB/B \in G/B$ is a $T$-representation with ($T$-equivariant) Euler class $e_w = e^T(T^*_{wB/B}G/B)$. For $N_w = T_{(wB/B, 0)}(X)$ the tangent space of $(wB/B,0)$ in $X = T^*(G/B)$, consider the $\T$-invariant decomposition $N_w = N_{w,+} \oplus N_{w, -}$ into $T$-weights that pair positively and negatively for our choice of dominant $\beta:S^1 \to T$, respectively. (There are no $0$-weights, as $\beta$ is chosen generically). By the duality from holomorphic symplectic form, one concludes that $N_{w, +} \cong N_{w, -}^\vee$ up to a twist of a character of the $S^1$-action rotating the holomorphic symplectic form. 

Therefore, up to a sign, the $T$-equivariant Euler class of $N_w$ given by $e(N_w) = e(N_{w, -})^2$ is a perfect square; from the polarization $T_{(wB/B,0)}(X) = T_{wB/B} G/B \oplus T^*_{wB/B}G/B$, we also have $e(N_w) = e_w^2$. Now fix the sign $\pm e(N_{w, -})$ so that 
\begin{equation}\label{eqn:polarization}
    \pm e(N_{w, -})|_{H^*_T(\mathrm{pt})} = e_w.
\end{equation}

Below, we consider the isolated fixed point $(wB/B,0) \in X^T$ corresponding to a Weyl group element $w\in W$, and denote by $[w] \in H^0(X^T)$ the corresponding fundamental class of the isolated point. For a class $\alpha \in H^*_\T(X)$, we define $\alpha|_w$ as the image of $\alpha$ under $H^*_\T(X) \to H^*_\T(X^T) \cong \bigoplus H^*_\T(\mathrm{pt}) \cdot [w] \to H^*_\T(\mathrm{pt})\cdot[w]$, i.e. restriction to the fixed locus followed by the projection to the summand generated (as a $H^*_\T(\mathrm{pt})$-module) by $[w]$.

\begin{defn}[{\cite[Definition 3.3.4]{MO19}, \cite[Definition 2.1]{SZ20}}]\label{defn:stable-envelopes}
    There exists a unique map of $H^*_{\T}(\mathrm{pt})$-modules
    \begin{equation}
        \mathrm{Stab}_+: H^*_\T (X^T) \to H^*_\T(X)
    \end{equation}
    satisfying the following properties.  (i) $\mathrm{Stab}_+(w) |_{H^*_{\T}(X \setminus \mathbb{Y})} = 0$, (ii) $\mathrm{Stab}_+(w)|_{w} = \pm e(N_{w, -})$ where the sign is determined by \eqref{eqn:polarization}, (iii) $\mathrm{Stab}_+(w)|_v$ is divisible by $\seq$ for $v < w$ in the Bruhat order. 
\end{defn}

\begin{rem}\label{rem:stable-envelope-properties}
Property (i) with \cref{prop:conormals-span-cohomology} implies that the stable envelope classes $\mathrm{Stab}_+(w)$ for $w \in W$ are linear combinations of the conormal cycles $Y_w$ in $H^*_\T(X)$. Property (ii) implies that the stable envelope classes form a basis in $H^*_\T(X)_{\mathrm{loc}}$ after localization; note that $e_w \neq 0$, as we are considering the Euler class of the normal bundle of the fixed locus (more precisely, its half).

Moreover, the stable envelopes are defined via Lagrangian correspondences (\cite[Proposition 3.5.1]{MO19}), and hence takes middle degree to middle degree. As $w \in W$ has a corresponding class $[w]:=[(wB/B,0)] \in H^0_\T(X^T)$, the stable envelope class $\mathrm{Stab}_+(w) \in H^{\dim_{\mathbb{C}}X}_\T(X)$ lies in middle dimension (which is also evident from property (ii)).
\end{rem}

There is also a map $\mathrm{Stab}_- : H^*_\T(X^T) \to H^*_\T(X)$ of the same kind such that the classes $\mathrm{Stab}_-(w)$ are linear combinations of the opposite conormals $Y^w$. The following is the key property we will use:

\begin{prop}[{\cite[Theorem 4.4.1]{MO19}}]\label{prop:stable-dual}
    The stable envelope classes and their opposites $\mathrm{Stab}_\pm (w)$ are intersection duals for the Poincar\'e pairing in $H^*_\T(X)_{\mathrm{loc}}$, that is 
    \begin{equation}
        \left(\mathrm{Stab}_+(w), \mathrm{Stab}_-(v) \right) = (-1)^{\dim_{\mathbb{C}}(G/B)} \delta_{v, w}.
    \end{equation}
\end{prop}

In particular, to define the quantum Steenrod operations and shift operators on $H^*_\T(X)_{\mathrm{loc}}$, it suffices to define them on the submodule of $H^*_\T(X)$, before localization, spanned by the stable envelope classes (equivalently the conormal cycles) and extend $H^*_\T(\mathrm{pt})_{\mathrm{loc}}$-linearly.

\begin{exmp}
    In the case of $T^*\mathbb{P}^1$, the support of $\mathrm{Stab}_+$ are given by the (union of) Lagrangians given by (i) the cotangent fiber at the north pole for $1 \in W$ and (ii) the union of the cotangent fiber at the north pole with the zero section $\mathbb{P}^1$ for $s \in W$.
\end{exmp}

\subsection{Regularity of moduli spaces for quantum Steenrod operations}\label{ssec:moduli-regularity-qst}
In this section, we show that the quantum Steenrod operations can be defined on $X = T^*(G/B)$ despite its non-exact and non-compact geometry. We define the quantum Steenrod operations by (i) first defining them on a submodule of $H^*_{\T}(X)$, generated by the stable envelopes which are preserved under quantum Steenrod operations, and (ii) showing that these relative cycles form a basis after passing to $H^*_{\T}(X)_{\loc}$, and hence extending the definition by linearity.

\begin{rem}\label{rem:ritter-zivanovic-regularity}
    The example of the Springer resolution $T^*(G/B)$ falls in the general class of manifolds considered in \cite{RZ23}, dubbed \emph{symplectic $\mathbb{C}^*$-manifolds} in \emph{loc. cit.} In \cite{RZ23}, a general maximum principle for such manifolds is established, which we expect to be adaptable to justify the constructions of quantum Steenrod and shift operators given in \cref{sec:operators} for a general class of conical symplectic resolutions. We have opted to establish the result directly for completeness, but the key idea of using the projection to the nilpotent cone $T^*(G/B) \to \mathcal{N}$ is the same as in the general approach of \cite{RZ23}.
\end{rem}

\subsubsection{Compactly intersecting incidence constraints}

Following the remark after \cref{prop:stable-dual}, we will show that the quantum Steenrod operations can be defined on the submodule of $H^*_\T(X)$ spanned by the image of Poincar\'e duals of conormal cycles $[Y_w]$. We must show the compactness and transversality of the moduli spaces with incidence constraints on the conormal cycles which define the structure constants. 

The difference from the previous general setup is that on the submodule of $H^*_\T(X)$ spanned by the conormal cycles, we have a control of non-compactness of the incidence cycles at $ 0\in C$,  as the conormal cycles approach infinity at prescribed directions (the conormal directions to the stable manifolds $\mathring{X}_w$). Consequently, we are not restricted to considering only the compact cycles as incidence constraints at $\infty \in C$, but we can consider locally finite cycles as long as they properly intersect the conormal cycles. The following is the key lemma for establishing this.

\begin{lem}\label{lem:incidence-compact-intersection}
    Consider the conormal cycles $Y_w \subseteq T^*(G/B)$ and the opposite conormal cycles $Y^v$. Their image under the projection $T^*(G/B) \to \mathcal{N}$ intersect in a compact subset.
\end{lem}
\begin{proof}
    Recall that $\mathbb{Y}$ is defined as the union of the conormal cycles $Y_w = \bigsqcup_{v \le w} N^*\mathring{X}_v$. The Schubert cells $\mathring{X}_v$ are defined as stable manifolds for the Morse function induced by the (real) moment map $G/B \to \mathbb{R}$ induced by the choice of the dominant cocharacter $\beta:S^1 \to T$. If one considers the corresponding dominant $\mathbb{C}^* \le T_\mathbb{C}$-action acting on $T^*(G/B)$, the conormal bundle $N^* \mathring{X}_w$ can therefore be identified with all points in $T^*(G/B)$ that converges to the fixed point $w$ under the dominant $\mathbb{C}^*$-action (as $z \in \mathbb{C}^*$ approaches $z \to \infty$).
    
    Consider the complex moment map $\mu: T^*(G/B) \to \mathfrak{g}^*$ which is identified with the Springer resolution, \cref{defn:springer-resolution}, with the embedding of the image $\mathcal{N} \subseteq \mathfrak{g}^*$. Note that $\mathfrak{g}^*$ is a $T$-representation and the moment map $\mu$ is $T$-equivariant. Decompose $\mathfrak{g}^* = V_+ \oplus V_0 \oplus V_-$ for $\beta$-positive, $\beta$-null and $\beta$-negative $T$-weight spaces $V_+$, $V_0$ and $V_-$. By equivariance, the union of ideal boundary of $\mathbb{Y}$ is projected under $\mu$ to the ideal boundary of $V_+$. Similarly, the opposite conormals $\mathbb{Y}^-$ at infinity is projected to the ideal boundary of $V_-$. Since the ideal boundaries of $V_+$ and $V_-$ are disjoint, the intersection of the images of $Y_w$ and $Y^v$ under $\mu$ must intersect in a compact subset.
\end{proof}

\subsubsection{Maximum principle}
Fix a cohomology class $b \in H^*_\T(X;\mathbb{F}_p)$ Poincar\'e dual to a $\T$-invariant locally finite homology cycle $[Y]$ and take $b_0 = \mathrm{PD}[Y_w] \in H^*_\T(X;\mathbb{F}_p)$ and $b_\infty = \mathrm{PD}[Y^v] \in H^*_\T(X;\mathbb{F}_p)$ for $v, w \in W$. Take the product $\mathcal{Y} = Y_w \times Y^{\times p} \times Y^v \subseteq X \times X^p \times X$ and the corresponding incidence cycle $\mathcal{Y}^{eq}$. 

\begin{rem}
    To consider the conormal cycles $Y_w, Y^v$ as incidence cycles, note that the smooth loci of conormal cycles are unions of embedded submanifolds (the conormal bundles of the Schubert cells), whose image under inclusion can be compactified by strata of real codimension $\ge 2$.
\end{rem}

We first consider genuine holomorphic maps from $C$, the parametrized copy of $\mathbb{P}^1$ (see \cref{sssec:operators-notation-source}) with incidence constraints on $Y_w \times Y^{\times p} \times Y^v$.
\begin{lem}\label{lem:qst-holo-compactness}
    For the integrable complex structure $J$ on $T^*(G/B)$, consider the moduli space of holomorphic genus $0$ curves with incidence constraints
    \begin{equation}
        \mathcal{M}_{\mathrm{hol}, A} = \{ u : C \to X : \overline{\partial}_J u = 0, u_*[C]=A, \ u(0) \in Y_w, \ u(\infty) \in Y^v, \ u(z_1), \dots, u(z_p) \in Y\}
    \end{equation}
    for fixed degree $A \in H_2(X;\mathbb{Z})$. There is a compact subset of $X$ that contains the image of all maps in $\mathcal{M}_{\mathrm{hol}, A}$ and its stable map compactification.
\end{lem}
\begin{proof}
Note that maximum principle applies to the affine variety $\mathcal{N} \subseteq \mathfrak{g}^*$, and the holomorphic curves in $\mathcal{N}$ must be constant. The moment map $\mu: T^*(G/B) \to \mathcal{N}$ is holomorphic, so the holomorphic curves in $T^*(G/B)$ can only occur in the fiber of $\mu$. Hence the curves with incidence constraints on $Y_w$ and $Y^v$ lie in the preimage of $\mu(Y_w) \cap \mu(Y^v) \subseteq \mathcal{N}$. By \cref{lem:incidence-compact-intersection} this subset is contained in a compact subset $K \subseteq \mathcal{N}$, and the result follows for the subset $\mu^{-1}(K)$ from the properness of $\mu$. The bubbles must also lie in the fibers of $\mu$, and their image is connected to the image of the principal component, hence the bubbles are also contained in the same compact subset $\mu^{-1}(K)$.
\end{proof}

Hence without any introduction of perturbation of the almost complex structure or inhomogeneous terms, the moduli space is indeed compact. Using compactness of the genuine holomorphic curves, the regularity of the moduli space of solutions to the perturbed Cauchy--Riemann equation can be shown as follows. As usual in Floer theory, we prove regularity of the moduli spaces for generic choices of perturbation data; in our case, the equivariant perturbation data are chosen from
\begin{align}
    J^{eq} &= \{ J_{w'} \}_{w' \in ET} \in C^\infty (ET \times X; \mathrm{End}(TX)), \\
    \nu^{eq} &= \{ \nu_{w, w'}\} \in C^\infty (E\mathbb{Z}/p \times ET \times C \times X; \mathrm{Hom^{0,1}}(TC;TX)).
\end{align}

Fix the integrable perturbation data $(J^{eq}_0, \nu^{eq}_0) = (J, 0)$ where $J$ is the (constant over $ET$) integrable complex structure, and consider its neighborhood $\mathcal{S} \ni (J^{eq}_0, \nu^{eq}_0) $ in the above spaces (in the $C^\infty$ topology).

\begin{prop}\label{prop:qst-regularity}
    There exists a comeager subset $\mathcal{S}^{\mathrm{reg}} \subseteq \mathcal{S}$ of equivariant perturbation data $(J^{eq}, \nu^{eq})$ near $(J^{eq}_0, \nu^{eq}_0)$, such that for any $(J^{eq}, \nu^{eq}) \in \mathcal{S}^{\mathrm{reg}}$, the moduli space with incidence constraints $\mathcal{M}_A^{eq;i;2j} \pitchfork \mathcal{Y}^{eq;i;2j}$ (cf. \eqref{eqn:i/2j-th-eq-moduli-space-qst}) is regular. Moreover, the image of solutions are contained in a compact subset of $X = T^*(G/B)$ (and hence Gromov compactness applies).
\end{prop}
\begin{proof}
The equivariant moduli spaces are just parametrized versions of the non-equivariant construction, and indeed by the contractibility of the space of compatible almost complex structures $J$ and inhomogeneous terms $\nu$, there are no obstructions for inductively constructing equivariant perturbation data. It therefore suffices to show that the non-equivariant moduli spaces can be made regular and all solutions lie in a compact subset of $X$.

From \cref{lem:qst-holo-compactness}, there exists a compact subset $K_0 \subseteq X$ such that all holomorphic curves satisfying the incidence constraints lie in $K_0$. Fix a bounded open subset $U \supseteq K_0$ and consider only the perturbation data $(J, \nu)$ supported in $U$. For applying standard transversality and compactness results, it suffices to show that there is a compact subset $K_1 \supseteq U$ only depending on $U$ such that any solution for the perturbation data $(J, \nu)$ such that $\mathrm{supp}(J, \nu) \subseteq U$ must be contained in $K_1$.

Let $u: C \to X$ be a solution of $\overline{\partial}_J u = \nu$ satisfying the incidence constraints, and consider the projection $\mu \circ u : C \to \mathcal{N}$. We show that there exists a compact subset in $\mathcal{N}$ which contains the image of $\mu \circ u$ for all solutions $u$. Suppose otherwise, and take a curve $u$ such that $\mu \circ u$ has a nonzero intersection with the bounded subset $\mu(U) \subseteq \mathcal{N}$. By the support condition on $(J, \nu)$, the curve $\mu \circ u$ outside $\mu(U)$ is an unperturbed holomorphic curve. By maximum principle this implies that $\mu \circ u$ is a constant curve outside $\mu(U)$, and $u$ is contained in the fiber of $\mu$ outside $\mu(U)$. This contradicts \cref{lem:qst-holo-compactness}, since by assumption $u$ must be contained in the fiber of $\mu$ over $\mu(K_0) \subseteq \mu(U)$. Hence a priori there is a compact subset $\mathcal{N} \supseteq K' \supseteq \mu(U)$ such that all solutions $u$ of the perturbed equation satisfy $\mu(\mathrm{im}(u)) \subseteq K'$. Setting $K_1 = \mu^{-1}(K')$, which is compact by properness of $\mu$, yields the desired compact subset of $X$. 

Standard transversality results now apply that for generic $(J, \nu)$ the moduli space is regular, and moreover Gromov compactness applies. Since $c_1(X) = 0$, the compactifying strata in the stable map compactification all have real codimension $\ge 2$ for generic $(J, \nu)$, and the counts $\mathcal{M}_A^{eq;i;2j} \pitchfork \mathcal{Y}^{eq;i;2j}$ are well-defined. \end{proof}

\begin{rem}
    Generic transversality statements do not directly apply to spaces of $C^\infty$-sections; rather, one (i) first passes to a suitable Banach completion (such as Floer's $C^\infty_{\varepsilon}$-spaces) to which one applies the generic transversality results, and (ii) uses e.g. the Taubes trick (see \cite[Section 4.4.2]{Wen-lec}) to pass to genericity in the $C^\infty$-topology.
\end{rem}

\begin{cor}\label{cor:qst-definable}
    For $X = T^*(G/B)$ and any $b \in H^*_\T(X)$, the (equivariant) quantum Steenrod operations $\Sigma_b^\T$ can be defined on the submodule of $H^*_\T(X)$ generated by the Poincar\'e duals of the conormal cycles $[Y_w]$, and extended to a map
    \begin{equation}
        \Sigma_b^\T : H^*_\T(X;\Lambda)_\loc \zpeq \to H^*_\T(X;\Lambda)_\loc \zpeq
    \end{equation}
    after localization.
\end{cor}
\begin{proof}
    By \cref{defn:stable-envelopes}, \cref{prop:stable-dual} and the remarks following them, to define $\Sigma_b^\T$ on the localized equivariant cohomology it suffices to define it on the stable envelope classes, which are in turn linear combinations of the classes represented by the conormal cycles. When the incidence constraints at $0, \infty \in C$ are given by conormal cycles, \cref{prop:qst-regularity} shows that the corresponding moduli spaces required to define the quantum Steenrod operations on the submodule generated by the conormal cycles are indeed regular, hence the operation is well-defined.
\end{proof}

\subsection{Regularity of moduli spaces for shift operators}\label{ssec:moduli-regularity-shift}
We now show that the shift operators can also be defined on $X = T^*(G/B)$, see \cref{assm:shift-regular}. For necessary compactness properties of the moduli spaces to hold, we must assume a condition for the cocharacter (see \cref{defn:cochar-nonneg} below). To lift this assumption on cocharacter for the definition of shift operators, one must pass to localized equivariant cohomology and use the definition of shift operators via localization from the algebro-geometric construction.

\begin{defn}\label{defn:cochar-nonneg}
    Consider the action of $\T = T \times S^1$ on $X$ and on $\mathfrak{g}^*$, such that the moment map projection $\mu : X \to \mathfrak{g}^*$ is $\T$-invariant. Consider the weights of the $\T$-action on $\mathfrak{g}^*$. The cocharacter $\cochar: S^1 \to \T$ is \emph{non-negative} if it pairs non-negatively with all the weights.
\end{defn}

As in the case of quantum Steenrod operations, fix $v,w \in W$ and consider conormal cycle $Y_w$ and opposite conormal cycle $Y^v$. Both are $\T$-invariant and therefore may represent $b_0 = PD[Y_w] \in H^*_{\mathbb{Z}/p \times \T}(X |\rho_0)$ and $b_\infty = PD[Y^v] \in H^*_{\mathbb{Z}/p \times \T}(X |\rho_\cochar)$ to be incidence constraints for the bare shift operators $S_\cochar$ (\cref{defn:bare-shift}). For a fixed degree $A \in H_2(X;\mathbb{Z})$, consider the section class $\widetilde{A} = s_{\mathrm{min}} + (\iota_0)_*(A) \in H_2(E_\cochar;\mathbb{Z})$, see \cref{defn:eq-moduli-sections}.

\begin{lem}\label{lem:shift-holo-compactness}
    Fix a nonnegative cocharacter $\cochar: S^1 \to \T$. For the distinguished integrable complex structure on $E_\cochar$, consider the moduli space of holomorphic sections with incidence constraints
    \begin{equation}
        \mathcal{M}_{\mathrm{hol}, \widetilde{A}}(E_\cochar) = \{ u : C \to E_\cochar : \overline{\partial}_{\widetilde{J}} u = 0, \ u_*[C]=\widetilde{A}, \ u(0) \in Y_w, \ u(\infty) \in Y^v\}
    \end{equation}
    for fixed degree $A \in H_2(X;\mathbb{Z})$. There is a compact subset of $E_\cochar$ that contains the image of all maps in $\mathcal{M}_{\mathrm{hol}, A}(E_\cochar)$ and its stable map compactification.
\end{lem}
\begin{proof}
    Denote the weights of $\T = S^1 \times T = S^1 \times (S^1)^r$ by $\seq, \teq_1, \dots, \teq_r$. Decompose $\mathfrak{g}^*$ into its $\T$-weight spaces,
    \begin{equation}
        \mathfrak{g}^* = \bigoplus_{i=1}^{\dim_{\mathbb{C}} \mathfrak{g}^*} \mathbb{C}[f_i], \quad f_i := n_{i0} h + \sum_{j=1}^r n_{ij} \lambda_j \in (\mathrm{Lie}\T)^*, \ n_{ij} \in \mathbb{Z}.
    \end{equation}
    Denoting the components of $\beta$ by $\langle \beta, \lambda_j \rangle = \beta_j$, \cref{defn:cochar-nonneg} translates to the condition that $\langle \cochar, f_i \rangle = \sum_{j=0}^r n_{ij} \cochar_j \ge 0$ for all $i = 1, \dots, \dim \mathfrak{g}^*$. 
    Consider the associated bundle $\mathcal{O}_{\mathbb{P}^1}(-1) \otimes_{\cochar} \mathfrak{g}^* \to \mathbb{P}^1$ with fibers isomorphic to $\mathfrak{g}^*$, defined as
    \begin{equation}
        \mathcal{O}_{\mathbb{P}^1}(-1) \otimes_{\cochar} \mathfrak{g}^* := \{ ([z], y) \in \mathcal{O}(-1) \times \mathfrak{g} \}/([z], y) \sim ([z] \cdot \theta, \beta(\theta^{-1})\cdot y) \mbox{ for } \theta \in S^1
    \end{equation}
    using the cocharacter $\beta : S^1 \to \T$. Then the complex moment map $\mu: X \to \mathfrak{g}^*$ applied fiberwise over $\mathbb{P}^1$ yields a holomorphic map
    \begin{equation}
        \widetilde{\mu} : E_\cochar \to \mathcal{O}(-1) \otimes_{\cochar} \mathfrak{g}^* \cong \bigoplus_{i=1}^{\dim \mathfrak{g}^*} \mathcal{O} ( - \langle \cochar, f_i \rangle ),
    \end{equation}
    where the vector bundle $\mathcal{V} = \bigoplus \mathcal{O}(-\langle \cochar, f_i \rangle)$ is a direct sum of line bundles of nonpositive degrees. Under $\widetilde{\mu}$ any holomorphic section $u: C \to E_\cochar$ of $E_\cochar \to \mathbb{P}^1$ maps to a holomorphic section of $\mathcal{V} \to \mathbb{P}^1$. If $u$ satisfies the incidence constraints given by $Y_w$ and $Y^v$ at $0$ and $\infty$, then $\widetilde{\mu} \circ u$ will satisfy incidence constraints given by $\mu(Y_w)$ and $\mu(Y^v)$ at $0$ and $\infty$.
    
    It suffices to show that there is a compact subset of $\mathcal{V}$ which contains all holomorphic sections of $\mathcal{V} \to \mathbb{P}^1$ with incidence constraints given by $\mu(Y_w)$ and $\mu(Y^v)$. The only nonzero holomorphic sections can arise for the subbundle spanned by summands such that $\langle \cochar, f_i \rangle = \sum_{j=0}^r n_{ij}\cochar_j = 0$, which forms a trivial vector bundle over $\mathbb{P}^1$. Considering a nonzero section of this subbundle as a graph of a constant map $p: C \to \mathfrak{g}^*$, note that for the section to satisfy the incidence constraints the image of $p$ must lie in the intersection $\mu(Y_w) \cap \mu(Y^v)$, which is compact by \cref{lem:incidence-compact-intersection}. The result follows from the properness of $\mu$ (it is the Springer resolution, see \cref{defn:springer-resolution}).
\end{proof}

Below, as in the proof of \cref{prop:qst-regularity}, we consider equivariant perturbation data $(\widetilde{J}^{eq}, \nu_E^{eq})$ in the neighborhood $\mathcal{S}$ (in the $C^\infty$-topology) of the reference perturbation data $(\widetilde{J}^{eq}_0, \nu_{E,0}^{eq}) = (\widetilde{J}, 0)$. A \emph{generic choice} is interpreted as a choice in a comeager subset of such neighborhood.

\begin{prop}[cf. {\cref{assm:shift-regular}}]\label{prop:shift-regularity}
    Fix a cocharacter $\cochar: S^1 \to \T$ satisfying \cref{defn:cochar-nonneg}. For a generic choice of $(\widetilde{J}^{eq}, \nu^{eq}_E)$, the moduli space with incidence constraints $\mathcal{M}_A^{eq;i;2j}(E_\cochar) \pitchfork \mathcal{Y}^{eq;i;2j}(E_\cochar)$ is regular. Moreover, the image of solutions are contained in a compact subset of $E_\cochar$.
\end{prop}
\begin{proof}[Proof sketch]
    The proof is essentially the same as in \cref{prop:qst-regularity}. Via \cref{lem:shift-holo-compactness}, one can choose a compact subset of $E_\cochar$ containing all holomorphic curves satisfying the incidence constraints, and consider perturbation data $(\widetilde{J}, \nu_E)$ supported in that bounded region. One applies the fiberwise moment map $\widetilde{\mu}: E_\cochar \to \mathcal{V} = \mathcal{O}(-1) \otimes_\cochar \mathfrak{g}^*$ (see the notation from \cref{lem:shift-holo-compactness}), and apply maximum principle in $\mathcal{V}$ to obtain an a priori bound for the image of perturbed solutions under $\widetilde{\mu}$. By properness of $\widetilde{\mu}$, the preimage gives the desired compact subset of $E_\cochar$ containing all perturbed solutions. Then standard transversality techniques and Gromov compactness apply.
\end{proof}

\begin{cor}[cf. {\cref{cor:qst-definable}}]\label{cor:shift-definable}
    For $X = T^*(G/B)$ and any $b \in H^*_\T(X)$, the shift operators $\mathbb{S}_\cochar$ for nonnegative cocharacters $\cochar: S^1 \to \T$ can be defined on the submodule of $H^*_\T(X)$ generated by the Poincar\'e duals of the conormal cycles $[Y_w]$, and extended to a map
    \begin{equation}
        \mathbb{S}_\cochar : H^*_{\mathbb{Z}/p \times \T}(X;\Lambda)_\loc \to H^*_{\mathbb{Z}/p \times \T}(X;\Lambda)_\loc
    \end{equation}
    after localization.
\end{cor}
\begin{proof}
    \cref{prop:shift-regularity} shows that the structure constants for incidence constraints given by the conormal cycles are well-defined.
\end{proof}

For later purposes, it is important to consider the inverses of the shift operators $\mathbb{S}_\cochar$, which may be defined as the reverse shift operators $\mathbb{S}_{-\cochar}$; however, the cocharacters $-\cochar$ would not satisfy nonnegativity condition, \cref{defn:cochar-nonneg}. Indeed, the structure constants of the inverse would involve poles in the equivariant parameters due to the non-compactness of the relevant moduli spaces. To bypass the need to discuss invertibility separately we will impose the following assumption in our discussion. The following assumption is essentially a comparison of algebro-geometric and symplecto-geometric definitions of Gromov--Witten invariants. For semipositive targets such as $X$, it is expected to be true.

\begin{assm}\label{assm:shift-ag-equals-sg}
    For nonnegative cocharacters $\cochar: S^1 \to \T$, the shift operators $\mathbb{S}_\cochar$ for $X$ agree with the shift operators as in \cite{BMO11}, \cite{Iri17}, defined for smooth complex varieties with algebraic torus actions.
\end{assm}

The main advantage of the algebro-geometric definition is that it allows the definition of Gromov--Witten invariants even when the moduli space is not compact, as long as the torus fixed points of the moduli space is compact, via equivariant localization.

\begin{cor}\label{cor:shift-invertible}
    The shift operators $\mathbb{S}_\cochar: H^*_{\mathbb{Z}/p \times \T} (X;\Lambda) \to H^*_{\mathbb{Z}/p \times \T}(X;\Lambda)$ are invertible after localization, i.e. passing to $H^*_{\mathbb{Z}/p \times \T}(X;\Lambda) \otimes \mathrm{Frac}(H^*_{\T}(\mathrm{pt}))$. Hence the same is true for the bare shift operators $S_\cochar$.
\end{cor}
\begin{proof}
    The invertibility follows from the relation $\mathbb{S}_{\cochar_1} \circ \mathbb{S}_{\cochar_2} = \mathbb{S}_{\cochar_1 + \cochar_2}$, as in \cite{Iri17}; one can take $\mathbb{S}_{-\cochar}$ as the inverse of $\mathbb{S}_\cochar$. Although $-\cochar$ would not be nonnegative if $\cochar$ is nonnegative, the corresponding shift operator can be defined on $H^*_{\mathbb{Z}/p \times \T}(X)_{\mathrm{loc}}$ by equivariant localization, see \cite[Remark 3.10]{Iri17} and \cite[Section 6]{BMO11}. 
\end{proof}

\subsection{Beyond Springer resolutions}\label{ssec:moduli-beyond-springer}
    The key to regularity properties of the moduli spaces involved in the definition of quantum Steenrod operations and shift operators are the (i) existence of compactly intersecting incidence cycles, provided by the conormal cycles or equivalently the stable envelope classes, and (ii) existence of projection to an affine base satisfying maximum principle. In this subsection, we briefly sketch how to extend our arguments for regularity to a wider generality of target varieties satisfying these properties.

\begin{defn}\label{defn:symplectic-resolution}
    A \emph{symplectic resolution} is a smooth complex variety $X$ with a holomorphic symplectic form $\omega_{\mathbb{C}}$ such that the affinization map $X \to X_0 = \mathrm{Spec}H^0(X, \mathcal{O})$ is a resolution of singularities (birational and proper). A symplectic resolution is \emph{conical} if it is equipped with a $\mathbb{C}^\times$ action scaling the holomorphic symplectic form with positive weight.
\end{defn}
The Springer resolution $X = T^*(G/B) \to X_0 = \mathcal{N}$ is the classical example. For a more thorough treatment, see the excellent \cite{Kal09}. To discuss quantum Steenrod operations and shift operators, we further require the following.
\begin{assm}\label{assm:equivariance}
    The conical symplectic resolution $X$ is equipped with an action of a torus $T$ preserving the holomorphic symplectic form, commuting with the $\mathbb{C}^\times$ action, such that the torus fixed points $X^T$ are discrete. The proper map $X \to X_0$ must be $\T = T \times \mathbb{C}^\times$-equivariant, and $X_0^T$ must be a single point.
\end{assm}
    Examples which satisfy \cref{assm:equivariance} include the Springer resolution $X = T^*(G/B)$ and its parabolic generalizations $T^*(G/P)$, and quiver vareities of type $A$. 

    The stable envelopes, satisfying the properties \cref{defn:stable-envelopes}, \cref{rem:stable-envelope-properties}, exist in a much wider generality of conical symplectic resolutions, see \cite[Section 3]{MO19} and \cite{SZ20}. In particular, for a choice of a generic cocharacter $\cochar: S^1 \to T$ (and a choice of a polarization data, see \cite[Section 3.3.2]{MO19}) one obtains the corresponding stable envelope map
    \begin{equation}
        \mathrm{Stab}_{\cochar}: H^*_\T(X^T) \to H^*_\T(X).
    \end{equation}
    Under \cref{assm:equivariance}, \cref{rem:stable-envelope-properties} applies and the stable envelope classes $\mathrm{Stab}_\cochar(x) \in H^{\dim_{\mathbb{C}} X}_\T(X)$ for isolated fixed points $x \in X^T$ form a basis in $H^*_\T(X)_{\mathrm{loc}}$. Moreover, the dual basis is given by the opposite stable envelopes $\mathrm{Stab}_{-\cochar}$.

    \begin{lem}[{\cite[Theorem 4.4.1]{MO19}}]\label{lem:stable-intersect-compactly}
        The intersection of cycles representing stable envelope classes $\mathrm{Stab}_\cochar(x)$ and opposite stable classes $\mathrm{Stab}_{-\cochar}(x')$ for $x, x' \in X^T$ is compact, and the image of the intersection under $X \to X_0$ lies in the $T$-invariant locus of $X_0$.
    \end{lem}
    \begin{proof}[Proof sketch]
        See \cite{MO19} for the proof, while we only record the key idea. For the fixed cocharacter $\cochar: \mathbb{C}^\times \to T$ and $x \in X^T$, consider the attracting set $\mathrm{Attr}_\cochar(x) = \{ y \in X : \lim_{t \to 0} \cochar(t) \cdot y = x \in X^T \}$. Take the union of attracting sets $\mathrm{Attr}_\cochar = \bigcup_{x \in X^T} \mathrm{Attr}_\cochar(x)$ over all $T$-fixed points. The definition and construction of the stable envelopes (see \cite[Section 3.5]{MO19}) show that the stable envelope cycles in $X$ are supported in $\mathrm{Attr}_\cochar$. Fixing a $T$-equivariant embedding $X_0 \to V$ into a $T$-representation (compare $\mathcal{N} \to \mathfrak{g}^*$ for $X = T^*(G/B)$), $\mathrm{Attr}_\cochar$ projects to non-negative weight spaces of $V$. Similarly, the support of opposite stable envelope cycles projects to non-positive weight spaces of $V$. Hence the projected image of the intersection of stable envelope classes must be contained in the $T$-invariant locus of $X_0$; compactness follows from properness of $X \to X_0$. 
    \end{proof}

    \begin{cor}[cf. {\cref{cor:qst-definable}, \cref{cor:shift-definable}}]\label{cor:symplectic-res-definable}
        Let $X$ be a conical symplectic resolution satisfying \cref{assm:equivariance}. Then the equivariant quantum Steenrod operations $\Sigma_b^\T$ and $\mathbb{S}_\cochar$ for nonnegative $\cochar : S^1 \to \T$ may be defined on $H^*_\T(X)_{\mathrm{loc}}$.
    \end{cor}
    \begin{proof}[Proof sketch]
    As in the case of \cref{lem:incidence-compact-intersection}, \cref{lem:stable-intersect-compactly} is sufficient to show that the curves in moduli space of holomorphic curves with incidence constraints on stable envelope cycles and their duals are contained in a compact subset of $X$. Using the proper, holomorphic projection to the affine base $X \to X_0$, one can proceed as in the proof of \cref{prop:qst-regularity} and \cref{prop:shift-regularity} to show the regularity of the moduli spaces of perturbed solutions. The nonnegativity condition on cocharacter from \cref{defn:cochar-nonneg} modifies so that the cocharacter pairs nonnegatively with $\T$-weights of the $\T$-representation $H^0(X;\mathcal{O})$.        
    \end{proof}

\section{$p$-curvature of equivariant quantum connections}\label{sec:p-curvature}
In this section, we review the equivariant quantum connections of conical symplectic resolutions from \cite{BMO11}, \cite{MO19}. The equivariant quantum connections of symplectic resolutions are special in that (i) they admit lifts to the integers $\mathbb{Z}$ (or with bounded denominators) in a suitable basis, (ii) dependence on Novikov parameters is a rational function (a priori, it is only a power series) and (iii) they admit a compatible collection of intertwining operators given by the shift operators. Such connections fall into the general class of flat connections studied in \cite{EV23a, EV23b}, in which structural results about their $p$-curvature is proven.

\subsection{Equivariant quantum connections}\label{ssec:pcurv-eq-qconn}
We are interested in the equivariant quantum connections of conical symplectic resolutions, considered in \cite{BMO11} and \cite{MO19}. For the running example of $X = T^*(G/B)$, the formula for the connection in purely representation-theoretic terms was explicitly worked out by \cite{BMO11}.

Recall that $\teq, \seq$ represent the equivariant parameters for the action of torus $T \le \T$ on $X$ (preserving the holomorphic symplectic form), and the action of $S^1 \le \T$ on $X$ (scaling the holomorphic symplectic form), respectively. We consider $H^*_\T(X)_{\mathrm{loc}}$ as a $H^*_\T(\mathrm{pt})_{\mathrm{loc}}$-module with basis given by the stable envelope classes $\mathrm{Stab}_+(w)$ indexed by Weyl group elements $w \in W$, see \cref{defn:stable-envelopes}. This basis of (localized) equivariant cohomology of $X$ is referred to as the \emph{stable basis}.

\begin{thm}[{\cite[Theorem 1.1]{BMO11}}]\label{thm:eq-qconn-are-qm}
    Consider the equivariant quantum connection operators $\nabla_b : H^*_{\mathbb{Z}/p \times \T}(X;\Lambda) \to H^*_{\mathbb{Z}/p \times \T}(X;\Lambda)$, indexed by $b \in H^2_\T(X;\mathbb{Z})$, of the Springer resolution $X = T^*(G/B)$. The connection satisfies the following properties:
    \begin{enumerate}[topsep=-2pt, label=(\roman*)]
        \item The structure constants of the quantum multiplication in the stable basis are integers.
        \item The dependence on the Novikov parameters is a rational function.
        \item The quantum connection commutes with the shift operators $\mathbb{S}_\cochar :H^*_{\mathbb{Z}/p \times \T}(X;\Lambda) \to H^*_{\mathbb{Z}/p \times \T}(X;\Lambda) $ for every choice of $\cochar$.
    \end{enumerate}
\end{thm}
\begin{proof}
    By linearity of quantum product with respect to equivariant parameters, it suffices to consider $b \in H^2(X;\mathbb{Z})$. Let $b \in H^2(X;\mathbb{Z})$ be identified with some weight of the maximal torus $T$. Denote by $R$, $R^\vee$ the set of roots and coroots and identify positive coroots (after a choice of a dominant chamber) $\alpha^\vee \in R_+^\vee \subseteq H_2(X;\mathbb{Z})$ with curve degree classes. The formula from \cite[Theorem 1.1]{BMO11} characterizes the quantum multiplication operator as
    \begin{equation}\label{eqn:BMO-formula}
        b \ \ast_{\T} = b \cupprod_\T + \ \seq \sum_{\alpha^\vee \in R^\vee_+} (b, \alpha^\vee) \frac{q^{\alpha^\vee}}{1- q^{\alpha^\vee}} ([s_\alpha ] - 1),
    \end{equation}
    where $[s_\alpha]$ denotes the Steinberg action on $H^*_\T(X)$ by the simple reflection $s_\alpha \in W$ corresponding to the root $\alpha$. 

    Property (ii) is evident from \eqref{eqn:BMO-formula}. Property (i) follows from the computation of the action of $[s_\alpha]$ in stable basis \cite[Lemma 3.2]{Su17} which we reproduce below:
    \begin{equation}\label{eqn:Su-formula}
        [s_\alpha] \cdot \mathrm{Stab}_+(w) = - \mathrm{Stab}_{+}(w) - \mathrm{Stab}_{+} (w s_\alpha),
    \end{equation}
    for $w \in W$. Property (iii) is an application of \cref{thm:shift-compatibility-qconn}, see also \cite[Section 6]{BMO11}.
\end{proof}

\begin{exmp}
    For $X = T^*\mathbb{P}^1$, the non-equivariant cohomology $H^*(X)$ is of rank $2$; the equivariant cohomology $H^*_\T(X)$ has two stable basis elements of cohomological degree $2$ given by $\mathrm{Stab}_+(1)$ and $\mathrm{Stab}_+(s)$ indexed by the Weyl group $W = \{ 1, s\}$. The action of the Weyl group is given as in \cref{eqn:Su-formula}.

    The precise form of the connection (in $S^1$-equivariant cohomology, where we forget the $T=S^1$-action) in the stable basis is given in \cite[Lemma 7.1]{Lee23}.
\end{exmp}

The form of the quantum multiplication \eqref{eqn:BMO-formula} is summarized by the following.
\begin{lem}\label{lem:qconn-form}
    In the stable basis, the quantum multiplication operator $b \ \ast_{\T}$ can be written as a sum of three matrices
    \begin{equation}
        b \ \ast_{ \T} = \left(h B^{(0)}_{\mathrm{cl}} + B^{(1)}_{\mathrm{cl}} \right) + \seq B_q^{(0)}
    \end{equation}
    where $B^{(0)}_\mathrm{cl}, B^{(1)}_\mathrm{cl}, B_q^{(0)}$ are matrices with entries in $H^*_\T(\mathrm{pt}; \mathbb{F}_p) \otimes \Lambda$, homogeneous in cohomological degree with degrees $0, 2, 0$, respectively. Moreover, $B_q^{(0)}|_{q^A = 0} = 0$.
\end{lem}
\begin{proof}
    Writing the matrix of $b \ \ast_{\T}$ in stable basis as a sum of the matrices of the classical (equivariant) cup product operator and the quantum part gives
    \begin{equation}
        b \ \ast_{\T} = b \cupprod_{\T} + \ b  \ast_{\mathrm{quantum}, \T}.
    \end{equation}
    Since all basis elements of the stable basis lie in the same cohomological degree $\dim_{\mathbb{C}} X$ (see \cref{rem:stable-envelope-properties}), and the quantum product $b \ast_\T$ is an operator of degree $2$, all entries of the matrix of $b\ \ast_{ \T}$ must be expressions of cohomological degree $2$. These expressions depend polynomially on the equivariant parameters $\seq, \teq$ of the $\T$-action of $X$; the regularity of moduli spaces ensure that the equivariant parameters appear with non-negative powers, and the Novikov parameter $q^A$ has degree zero (as $c_1=0$). Expanding in $\seq$, one can write in the stable basis $b \cupprod_{\T} = \seq B^{(0)}_\mathrm{cl} + B^{(1)}_{\mathrm{cl}}$ and $b \ \ast_{\mathrm{quantum}, \T} = \seq B_q^{(0)}$. Note that all entries of the purely quantum part is divisible by $\seq$, which is evident from \eqref{eqn:BMO-formula} and is a consequence of the geometric fact that non-$S^1$-equivariant Gromov--Witten invariants of $X$ vanish.
\end{proof}

\subsection{$p$-curvature}\label{ssec:pcurv-defn}
In particular, the equivariant quantum connection of the Springer resolution admits a mod $p$ reduction. The mod $p$ equivariant quantum connection was the subject of study in \cite{Lee23}. To any algebraic connection in characteristic $p$, classically one associates the \emph{$p$-curvature} endomorphism. 

For the quantum connection, the $p$-curvature endomorphism takes the following specific form.

\begin{defn}\label{defn:p-curvature}
    Fix $b \in H^2_\T(X;\mathbb{F}_p)$ a mod $p$ reduction of an integral class, and consider the corresponding equivariant quantum connection $\nabla_b = t \partial_b + b \ \ast_{ \T} : H^*_\T(X;\Lambda)\zpeq \to H^*_\T(X;\Lambda)\zpeq$, see \cref{defn:eq-qconn}. The \emph{$p$-curvature} in the direction of $b$ is the degree $2p$ endomorphism of $H^*_\T(X;\Lambda)\zpeq$ defined as
    \begin{equation}\label{eqn:pcurv-defn}
        F_b = \nabla_b ^p - t^{p-1} \nabla_b.
    \end{equation}
\end{defn}
\begin{rem}
    The usual notion of the $p$-curvature for a connection measures the failure of the map $\partial \mapsto \nabla_\partial$ from derivations (on the base of connections) to bundle endomorphisms to be a map of \emph{restricted} Lie algebras in positive characteristic, $F_\partial := \nabla_\partial^p - \nabla_{\partial^p}$; compare this with the usual notion of the curvature $F_{\partial, \partial'} = [\nabla_{\partial} , \nabla_{\partial'}] - \nabla_{[\partial, \partial']}$, which measures the failure of the $\nabla$ to be a map of Lie algebras. For our example, we consider the $p$-curvature for the Novikov differentiation $t \partial_b$ whose $p$th power is $t^{p-1}(t\partial_b)$.
\end{rem}

Since the $p$-curvature is only defined in terms of the quantum connection, many of its properties immediately follow from that of the quantum multiplication.

\begin{lem}\label{lem:pcurv-properties}
  For $b \in H^2(X;\mathbb{F}_p)$, the $p$-curvature $F_b$ satisfies the following properties.
  \begin{enumerate}[topsep=-2pt, label=(\roman*)]
        \item $F_b(b_0)|_{(t,\theta) = 0} = \overbrace{b \ast_\T \cdots \ast_\T b}^{p} \ast_\T \ b_0$.
        \item $F_b(b_0)|_{q^A = 0} = \mathrm{St}^\T(b) \cupprod_{\T} b_0$, where $\mathrm{St}^\T(b) \in H^*_{\T}(X;\mathbb{F}_p)\zpeq$ is the class obtained by the power operation construction as in \eqref{eqn:classical-steenrod} in $\T$-equivariant cohomology.
        \item $F_b$ commutes with the shift operators $\mathbb{S}_\cochar$, and equivalently satisfies
        \begin{equation}
        F_b |_{\lambda \mapsto \lambda - \cochar t} \circ S_\cochar = S_\cochar \circ F_b.
    \end{equation}
        
  \end{enumerate}
\end{lem}
\begin{proof}
    Property (i) follows from that $\nabla_b|_{(t, \theta) = 0} = b \ \ast_\T$. Property (ii) follows from the definition of the $p$-curvature $F_b = \nabla_b^p - t^{p-1}\nabla_b$ together with the following two observations. (a) The computation of the classical Steenrod operations (denoted $P^0$ and $P^1$) gives 
    \begin{equation}
    \mathrm{St}^\T(b) = P^1(b) - t^{p-1} P^0(b) =  b^{\cupprod_\T p} - t^{p-1} b    
    \end{equation}
    for degree $2$ classes applied to $b \in H^2_\T(X;\mathbb{F}_p)$ (the Bockstein vanishes as we assume that $H^*(X)$ is torsion-free). The notation $b^{\cupprod_\T p}$ is the $p$-fold cup power of $b$ in $\T$-equivariant cohomology. (b) The connection satisfies $\nabla_b|_{q^A = 0} = t \partial_b + b \cupprod_\T$ and that $\partial_b$ commutes with the classical multiplication $b \cupprod_\T$, hence $\nabla_b^p|_{q^A = 0} = t^{p} \partial_b + b^{\cupprod_\T p}$. Property (iii) follows from the fact that $\nabla_b$ commutes with the shift operators, \cref{thm:shift-compatibility-qconn}.
\end{proof}

In \cref{sec:p-curvature-and-qst}, we will be interested in the eigenvalues of $F_b$. The understanding of the spectrum of $F_b$ fits in the framework of the recent work \cite{EV23b} of Etingof--Varchenko, where various structural properties are established for the (mod $p$ reductions of) a special class of flat connections arising from a \emph{periodic pencil of flat connections}. The input from the theory of periodic pencils can be bypassed for the running example $X = T^*(G/B)$, but we believe their precise description of the spectrum of $p$-curvature to be the key ingredient in the further study of equivariant quantum connections of symplectic resolutions in positive characteristic.

Exploiting the controlled rational function dependence on the Novikov parameters (cf. \cref{thm:eq-qconn-are-qm} (ii)), let us define the smaller Novikov ring of the form $\Lambda_{\mathrm{poly}} := k[q_i:= q^{A_i} ; A_i \in H_2(X;\mathbb{Z})]$ for a fixed algebraically closed field $k$ containing $R = \mathbb{F}_p$ and let $H = \mathrm{Spec}(\Lambda_{\mathrm{poly}})$. The connection is defined over a Zariski open subset of $H$, away from the locus where the poles appear.

\begin{defn}
    Let $\nabla$ be a flat connection (possibly with poles) defined over a sublocus $H_{\mathrm{reg}} := H \setminus \bigcup_{\alpha} \{ q^\alpha = 1\}$ on a trivial bundle $V \times H_{\mathrm{reg}} \to H_{\mathrm{reg}}$, which depends on a parameter $\hbar \in k$ \emph{linearly}, i.e. \begin{equation}
\nabla_i = \nabla_i(\hbar) : = q_i\partial_{q_i} + \hbar B_i(q).
    \end{equation} The connections $\nabla(\hbar)$ form a \emph{periodic pencil of flat connections} if there exists a gauge equivalence $S(\hbar) \in \Lambda_{\mathrm{poly}}(\hbar)[(q^{\alpha}-1)^{-1}] \otimes \mathrm{End}(V)$ of connections with rational function dependence on $q_i$, $\hbar$ such that
    \begin{equation}
        \nabla(\hbar - 1) \circ S(\hbar) = S(\hbar) \circ \nabla(\hbar).
    \end{equation}
\end{defn}

\begin{exmp}
    The ($S^1$)-equivariant quantum connection of the Springer resolution after the identification of parameters $\seq / t = \hbar$ forms a periodic pencil of flat connections, considered as a trigonometric connection, by the results of \cite{BMO11}. Namely, the quantum product has only a single power of $\seq$; we divide by $t$ as the quantum connection is only a connection in the usual convention after dividing by $t$. The bare shift operators $S_\cochar$ provide the required gauge equivalences; see \cite[Proposition 6.1]{BMO11} and \cref{thm:shift-compatibility-qconn}.
\end{exmp}

Let $\nabla$ be a periodic pencil of flat connections whose structure constants are defined over the integers $\mathbb{Z}$ so that it admits a mod $p$ reduction; more generally one can allow bounded denominators in the structure constants. Etingof--Varchenko \cite{EV23b} proves the following statement. We will not directly use this result so we only mention it in the following form, see \cite{EV23b} for details.

\begin{thm}\cite[Theorem 3.3]{EV23b}\label{thm:etingof-varchenko}
    Let the spectrum of the connection matrix of the connection $\nabla_i$ in a periodic pencil be given by $\Lambda_{ij}(q)$, $j = 1, \dots, \mathrm{rank}(V)$. Then the spectrum of the $p$-curvature $F_i$ are explicitly given by
    \begin{equation}
        (\hbar - \hbar^p) \Lambda_{ij}(q)^p.
    \end{equation}
    In particular, if the connection matrix of $\nabla_i$ has simple spectrum (distinct eigenvalues of multiplicity one), then so does the $p$-curvature $F_i$.
\end{thm}

In \cite[Section 7]{Lee23}, we have observed that for the quantum Steenrod operation of $T^*\mathbb{P}^1$, the specialization of equivariant parameters $\Sigma_b|_{\seq/t = \hbar}$ for $\hbar \in \mathbb{F}_p$ annihilates Varchenko's mod $p$ flat sections of the quantum connection. \cref{thm:etingof-varchenko} is consonant with this observation; specializing $\hbar \in \mathbb{F}_p$ annihilates all eigenvalues.


\section{$p$-curvature and quantum Steenrod operations}\label{sec:p-curvature-and-qst}
In this section, we prove our main theorem identifying the $p$-curvature of the equivariant quantum connection of symplectic resolutions with quantum Steenrod operations, for the case of the Springer resolution. The key ingredient is the compatibility of quantum Steenrod operations with the shift operators.

\subsection{A table}\label{ssec:table}
For the proof of the main theorem below, we use several properties of the (equivariant) quantum Steenrod operations, the $p$-curvature of the (equivariant) quantum connection, and the shift operators. Here we record, for the convenience of the reader, their linearity properties and respective limits under different specializations of equivariant parameters. We denote by $q^A$ the Novikov parameters, $(t,\theta)$ the $\mathbb{Z}/p$-equivariant parameters, $(\lambda, h)$ the $\T = T \times S^1$-parameters.

\begin{center}\label{tab:operators}
    \begin{tabular}{|c||c|c|c|c|c|}
    \hline
         operator & $q \mapsto 0$ & $(t,\theta) \mapsto 0$ & $h \mapsto 0$ & commutativity with $\nabla_a$ & linear over \\
     \hline\hline
        $\Sigma_b^\T$ & $\mathrm{St}^\T(b) \cup_\T $  & $b^{\ast_\T p} \ast_\T$ & $\mathrm{St}^T(b) \cupprod_T$ & $[\Sigma_b^\T, \nabla_a ] = 0$& $q,t, \theta, h, \lambda$ \\
        $F_b$ & $\mathrm{St}^\T(b) \cup_\T $ & $b^{\ast_\T p} \ast_\T$ & $\mathrm{St}^T(b) \cupprod_T$  & $[F_b,  \nabla_a ] = 0 $ & $q,t, \theta, h, \lambda$ \\
        $\mathbb{S}_\beta = \Phi_\beta \circ S_\beta$ &  twist & Seidel map& & $[\mathbb{S}_\beta , \nabla_a] = 0$ & $q, t, \theta$ \\
    \hline
    \end{tabular}
\end{center}
For the first and the second row, the corresponding claims can be found in \cref{ssec:operators-qst-teq} and \cref{ssec:pcurv-defn}.

The limit $h \mapsto 0$ for $\Sigma_b^\T$ and $F_b$ being $\mathrm{St}^T(b) \cupprod_T$ is specific to the target being a holomorphic symplectic manifold. A priori, the limit $h \mapsto 0$ of each operator corresponds to the $T$-equivariant (as opposed to $\T$-equivariant) quantum Steenrod operation $\Sigma_b^T$,  and the $p$-curvature of the $T$-equivariant quantum connection $(t\partial_b + b \ast_T )^p - t^{p-1} (t\partial_b + b \ast_T)$ (as opposed to $F_b = (t\partial_b + b \ast_\T )^p - t^{p-1} (t \partial_b + b \ast_\T)$), respectively.  If we work non-equivariantly with respect to the dilating $S^1$-action, the Gromov--Witten theory is trivial (see \cite[Section 1.3]{BMO11} for our target $X$, hence setting $h = 0$ has the effect of eliminating all $q \neq 0$ terms in the respective operators, reducing the limit to the first column $q \mapsto 0$.

For the last row, ``twist'' refers to the reparametrization morphism which sends $(t,\theta, \lambda,h) \mapsto (t, \theta, \lambda+\beta_\lambda t, h + \beta_h t)$ for the cocharacter $\beta = (\beta_\lambda, \beta_h) \in \mathrm{Lie}(T) \times \mathrm{Lie}(S^1) = \mathrm{Lie}(\T)$. The shift operator is only defined for $X$ with a torus action, hence the corresponding entry is empty. When we work only $\T$-equivariantly (as opposed to $\mathbb{Z}/p \times \T$-equivariantly) in its construction, the shift operators reduce to the Seidel representation \cite{LJ20, Iri17, Sei97}.

\subsection{Equivalence of $p$-curvature and quantum Steenrod operations}\label{ssec:pcurvqst-equivalence}
Take $X = T^*(G/B)$ with its action of $\T = T \times S^1$ from the maximal torus $T$ and the rotation $S^1$ of cotangent fibers. Fix $b \in H^2_\T(X;\mathbb{F}_p)$. We consider the $\T$-equivariant quantum connection $\nabla_b$ and the $\T$-equivariant quantum Steenrod operations $\Sigma_b^\T$. Denote by $F_b = \nabla_b^p - t^{p-1} \nabla_b$ the $p$-curvature in the direction of $b$.

For this section, with the $\T$-equivariance understood, we will simply denote $\Sigma_b^\T$ by $\Sigma_b$.

Recall that $t, \teq, \seq$ represent the equivariant parameters for $\mathbb{Z}/p$-action on the source curve $C$ (acting trivially on $X$), the action of torus $T \le \T$ on $X$ (preserving the holomorphic symplectic form), and the action of $S^1 \le \T$ on $X$ (rotating the holomorphic symplectic form), respectively.

Given the properties of the quantum Steenrod operations $\Sigma_b$ established previously, the following proof follows the structure of the proof of \cref{thm:etingof-varchenko} due to \cite{EV23b}.

\begin{thm}\label{thm:qst-is-pcurv-nilpt}
    The difference $F_b - \Sigma_b$ is a nilpotent operator. 
\end{thm}
\begin{proof}
    Consider both operators $F_b$ and $\Sigma_b$ as endomorphisms of $H^*_\T(X;\Lambda)_{\mathrm{loc}}\zpeq$ of degree $2p$. Fix the basis of $H^*_\T(X;\Lambda)_{\mathrm{loc}}\zpeq$ (as a $H^*_{\sg \times \T}(\mathrm{pt};\Lambda)_{\mathrm{loc}} = \Lambda(\!( \lambda, h)\!)[\![t, \theta ]\!]$-module) given by the stable envelopes, see \cref{defn:stable-envelopes}. Since the torus fixed locus $X^T$ are isolated points, the stable basis is homogeneous in cohomological degrees; that is, 
 $\mathrm{Stab}_+(w) \in H^{\dim_{\mathbb{C}} X}_{\T} (X)$ for all $w \in W$. 
    
    Therefore the structure constants of both operators in the stable basis are given by expressions of cohomological degree $2p$, which are polynomials in the equivariant parameters $t, \teq, \seq$ and power series in the Novikov parameters $q^A$.  Polynomiality follows from that $|q^A| = 0$ from $c_1 = 0$, and moreover that the regularity of the relevant moduli spaces ensure that the equivariant parameters only appear with non-negative powers.

    Hence one can expand the two matrices in $\seq$, yielding the sum
    \begin{align}
        F_b &= \seq^p F_b^{(0)} + \seq^{p-1} F_b^{(1)} + \cdots + \seq F_b^{(p-1)} + F_b^{(p)}, \\
        \Sigma_b &= \seq^p \Sigma_b^{(0)} + \seq^{p-1} \Sigma_b^{(1)} + \cdots + \seq \Sigma_b^{(p-1)} + \Sigma_b^{(p)}
    \end{align}
    where $F_b^{(i)}$, $\Sigma_b^{(i)}$ are matrices with entries given by expressions of cohomological degree $2i$ that are polynomials in the equivariant parameters $t, \teq$ and power series in Novikov parameters $q^A$.

    Write the quantum connection as $\nabla_b = t \partial_b + b \ast_{\T}$, and in the stable basis expand as (\cref{lem:qconn-form})
    \begin{equation}
        b \ \ast_{\ \T} = b \cupprod_{\T} + b \ \ast_{\mathrm{quantum}, \T} = (\seq B^{(0)}_{\mathrm{cl}} + B^{(1)}_{\mathrm{cl}}) + \seq B_q^{(0)}
    \end{equation}
    where $B^{(0)}_\mathrm{cl}$,  $B^{(1)}_\mathrm{cl}$, $B^{(0)}_q$ have entries of cohomological degrees $0, 2, 0$ respectively. Note that $B^{(1)}_{\mathrm{cl}}$ is simply the matrix of $b \cupprod_T$ in the equivariant cohomology for the smaller torus $T \le \T$.

    First note that $F_b^{(p)} = \Sigma_b^{(p)}$ from \cref{prop:eq-qst-properties} and \cref{lem:pcurv-properties}, as both matrices are equal to the matrix of the operator $\mathrm{St}^T(b) \cupprod_T$. This is due to the fact that $h=0$ limit corresponds to considering counts non-equivariantly with respect to the dilating $S^1$-action, for which the Gromov--Witten theory is trivial (cf. \cite[Section 1.3]{BMO11}).

    Also note that $F_b^{(0)} = \Sigma_b^{(0)}$, as both matrices are equal to the matrix of the $\seq^p$-coefficient of the operator $\overbrace{b \ast_\T \cdots \ast_\T b}^{p} \ast_\T$ for the following degree reasons. For $\Sigma_b$, note that $\Sigma_b = \overbrace{b \ast_\T \cdots \ast_\T b}^{p} \ast_\T + O(t)$ so the only terms divisible by $\seq^p$ must occur from the $p$-fold quantum multiplication. For $F_b = \nabla_b^p - t^{p-1} \nabla_b$, that the same is true can be seen from the expression $\nabla_b = t \partial_b + b \ast_{\T}$.

    Hence the difference $F_b - \Sigma_b$ is of the form
    \begin{equation}
        F_b - \Sigma_b = \seq^{p-1} \left(F_b^{(1)} - \Sigma_b^{(1)}\right) + \cdots + \seq \left(F_b^{(p-1)} - \Sigma_b^{(p-1)}\right).
    \end{equation}
    The characteristic polynomial of this matrix
    \begin{equation}
        \chi(F_b - \Sigma_b) = \sum_{i=0}^{n} (-1)^i \mathrm{tr} \left( \Lambda^i (F_b - \Sigma_b) \right) x^{n-i},
    \end{equation}
    where $\Lambda^i$ denotes the $i$th exterior product, satisfies the following properties: for each $1 \le i \le n$, 
    \begin{enumerate}[topsep=-2pt, label=(\roman*)]
        \item the coefficient $\mathrm{tr} \left( \Lambda^i (F_b - \Sigma_b) \right)$ is divisible by $\seq^i$;
        \item as a polynomial in $\seq$, the degree of the cofficient $\mathrm{tr} \left( \Lambda^i (F_b - \Sigma_b) \right)$ is bounded above by $(p-1)i$.
    \end{enumerate}

    Since $F_b$ and $\Sigma_b$ both commute with the shift operators from \cref{lem:pcurv-properties} and \cref{thm:compatibility}, so is the difference $F_b - \Sigma_b$. Consider the shift operator $S_\cochar$ for the cocharacter $\cochar : S^1 \to \T$ that vanishes on all $T$-weights and is weight $1$ for the $S^1$-action rotating the holomorphic symplectic form. By the invertibility of shift operators \cref{cor:shift-invertible}, we can conjugate $F_b - \Sigma_b$ by the shift operator which shows the invariance of the characteristic polynomial $\chi(F_b - \Sigma_b)$ under shift of equivariant parameters $\seq \mapsto \seq - t$.

    In particular, property (i) above implies that the coefficient $\mathrm{tr} \left( \Lambda^i (F_b - \Sigma_b) \right)$ is divisible by
    \begin{equation}
        \seq^i (\seq-t)^i \cdots (\seq-(p-1)t)^i = \left( \seq^p - t^{p-1} \seq \right)^i,
    \end{equation}
    which is a polynomial in $\seq$ of degree $pi > (p-1)i$. Hence by property (ii) above, the coefficient $\mathrm{tr} \left( \Lambda^i (F_b - \Sigma_b) \right)$ must vanish for all $1 \le i \le n$. This is the desired result.
\end{proof}

\begin{cor}\label{cor:qst-is-pcurv}
    For almost all $p$, the $p$-curvature and the quantum Steenrod operation for $b \in H^2_\T(X;\mathbb{F}_p)$ agree.
\end{cor}
\begin{proof}
    From \cref{thm:qst-is-pcurv-nilpt}, we know that the difference $F_b - \Sigma_b$ is a nilpotent operator acting on $H^*_{\T}(X;\Lambda)_{\mathrm{loc}}\zpeq$ as a $H^*_{\sg \times \T}(\mathrm{pt})_{\mathrm{loc}}$-module. In particular, it suffices to show that $\Sigma_b$ and $F_b$ are simultaneously diagonalizable. From \cref{thm:eq-cov-constancy}, $\Sigma_b$ commutes with $\nabla_a$ and hence with $F_a = \nabla_a^p - t^{p-1} \nabla_a$ for any $a \in H^2_\T(X;\mathbb{Z})$, and also $F_b$ and $F_a$ commute by the flatness of $\nabla$. In particular, it suffices to show that there exists $a \in H^2_\T(X;\mathbb{Z})$ for which $F_a$ has simple spectrum (i.e. diagonalizable with distinct eigenvalues) to show that $F_a = \Sigma_a$.

    To check that the eigenvalues of $F_a$ are all distinct for some $a$, we consider the discriminant of $F_a$, which we denote by $D(F_a)$. Since the equivariant parameter $t \in H^2_{\mathbb{Z}/p}(\mathrm{pt})$ is not inverted, one can consider the specialization $D(F_a)|_{t=0} = D(a^{\ast_\T p} \ast_\T)$; if the operator $a^{\ast_\T p} \ast_\T$ has simple spectrum, then the discriminant of it would be nonzero, and $D(F_a)$ would also be nonzero. Therefore it suffices to check that  $a^{\ast_\T p} \ast_\T$ has simple spectrum; this follows if $a {\ast_\T}$ has simple spectrum.
    
    Choose $a \in \mathrm{Lie}(T)^* \cong H^2(X;\mathbb{Z}) \subseteq H^2_\T(X;\mathbb{Z})$ which corresponds to a character $a : T \to S^1$. Consider the eigenvalues of $a \ast_\T|_{q^A =0} = a \cupprod_\T$; these are the weights of the (Borel--Weil) line bundle corresponding to $a$, restricted to $T$-fixed points of $X$. This line bundle is given by (the pullback under $T^*(G/B) \to G/B$ of the) associated line bundle $G \times^B \mathbb{C} \to G/B$, where we use the character $a : T \to S^1$ (and its unique extension to $B$) to make $\mathbb{C}$ into a $B$-representation. 
    
    At the point $B/B \in X^T$, this weight is tautologically given by $a$ itself, and at the other fixed points $wB/B \in X^T$, the weights are given by the Weyl group orbits $w \cdot a$ of $a$. When we choose $a$ generically (e.g. in the dominant Weyl chamber), these weights are indeed distinct in $H^*_{T}(\mathrm{pt};\mathbb{F}_p)$ for $X = T^*(G/B)$ for all but finitely many primes $p$, as they lie in different Weyl chambers. 
    
    Therefore $a \cupprod_{\T}$ has simple spectrum, which implies that $a \ast_\T$ has simple spectrum. Hence $F_a$ has simple spectrum, and both $F_b$ and $\Sigma_b$ which commute with $F_a$ can be simultaneously diagonalized into the eigenbasis of $F_a$, so the desired result follows.
\end{proof}

\begin{rem}\label{rem:beyond-springer}
    Proceeding as in \cref{ssec:moduli-beyond-springer}, this result also applies to further examples including $X = T^*(G/P)$ (the generalized Springer resolutions), quiver varieties of type $A$, or hypertoric varieties, which have (i) isolated $X^T$ and hence a stable basis of homogeneous cohomological degree, (ii) (jointly) simple spectrum of the classical (hence also quantum) multiplication operator. The Springer resolution $T^*(G/B)$ in type $A$ are in the intersection of the first two classes of examples. 

    There are examples such as $X = \mathrm{Hilb}^n(\mathbb{C}^2)$ with $X^T$ discrete, with semisimple quantum multiplication operator \cite[Section 4.1]{OP10}, such that the classical multiplication operator is not semisimple. For all but finitely many primes $p$, the result holds for these examples as well.
\end{rem}

\appendix

\bibliographystyle{amsalpha}
\bibliography{shift}

\end{document}